\documentclass[sn-mathphys]{sn-jnl}

 \usepackage{amssymb}
\usepackage{epstopdf}

\jyear{2021}%

\theoremstyle{thmstyleone}%
\newtheorem{theorem}{Theorem}
\theoremstyle{thmstyletwo}%
\newtheorem{remark}{Remark}%

\theoremstyle{thmstylethree}%
\newtheorem{definition}{Definition}%
\newtheorem{lemma}{Lemma}
\begin{document}

\title[Greedy capped nonlinear Kaczmarz methods]{Greedy capped nonlinear Kaczmarz methods}


\author[1]{\fnm{Yanjun} \sur{Zhang}}

\author*[1]{\fnm{Hanyu} \sur{Li}}\email{lihy.hy@gmail.com or hyli@cqu.edu.cn}

\affil[1]{\orgdiv{College of Mathematics and Statistics}, \orgname{Chongqing University}, \orgaddress{ \city{Chongqing}, \postcode{401331}, \state{P.R.}, \country{China}}}


\abstract{To solve nonlinear problems, 
we construct two kinds of greedy capped nonlinear Kaczmarz methods by setting a capped threshold and introducing an effective probability criterion for selecting a row of the Jacobian matrix. The capped threshold and probability criterion are mainly determined by the maximum residual and maximum distance rules. The block versions of the new methods are also presented. 
We provide the convergence analysis of these methods and their numerical results behave quite well.}



\maketitle

\section{Introduction}
\label{sec;introduction}
Consider the nonlinear system
\begin{align}
f(x)=0,  \label{Sec11}
\end{align}
where $f:\mathbb{R}^{n}\rightarrow\mathbb{R}^{m}$ and $x\in\mathbb{R}^{n}$ is an unknown variable. We assume throughout that $f(x)=[f_{1}(x), \cdots, f_{m}(x)]^T\in\mathbb{R}^{m}$ is a continuously differentiable vector-valued function, and there exists a solution $x_\star$ satisfying (\ref{Sec11}), i.e., $f(x_\star)=0$. Such problem is fundamental in numerical computing arising from many areas, e.g.,  machine learning \cite{chen2019homotopy}, differential equations \cite{hao2014bootstrapping} and optimization problems \cite{hao2018homotopy}.

There have been many studies on solving nonlinear problems by iterative methods \cite{dennis1977quasi,ortega2000iterative,yamashita2001rate}. Among them the nonlinear Kaczmarz method \cite{yuan2022sketched,zeng2020successive,wang2022nonlinear} is a typical representative for the so-called single sample-based method and can be formulated as
\begin{align}
x_{k+1}=x_{k}-\frac{f_{i_k}(x_k)}{\|\nabla f_{i_k}(x_k)\|^2_2}\nabla f_{i_k}(x_k).\label{Sec1.4}
\end{align}
In the nonlinear Kaczmarz iteration scheme, each iteration is formed by projecting the current point $x_k$ to the constraint set defined by $f_{i_k}(x_k)+\nabla f_{i_k}(x_k)^T(x-x_k)=0 $. 
After determining the iteration formula, it can be found that how to select the index $i_k$ is particularly important. At present, there are  mainly three different rules for selecting $i_k$ \cite{wang2022nonlinear} leading to three different methods: $(1)$ Nonlinear randomized Kaczmarz (NRK) method, where $i_k$ is randomly selected from $[ m ]$ with probability of
$
p_{i_k}= \frac{\lvert f_{i_k}(x_k)\rvert^2}{\|f(x_k)\|^2_2}
$; $(2)$ Nonlinear Kaczmarz (NK) method, where $i_k$ cyclically picks value from $[ m ]$; $(3)$ Nonlinear uniformly randomized Kaczmarz (NURK) method, where $i_k$ is randomly sampled from $[ m ]$ with equal probability. The theoretical analysis in \cite{wang2022nonlinear} shows that the convergence factors of the NRK and NURK methods are the same, however, the NRK method performs better in terms of the number of iterations and computing time in most cases.
It also usually outperforms 
the NK method. 
Upon close examination, we can find that the probability criterion used in the NRK method makes 
the method 
select the index corresponding to the larger entry of the residual vector as much as possible. However, it is possible for the NRK method to select an index with relatively small residual component at each iteration. This is inconsistent with the original intention of the NRK method and may lead to slower convergence.

In this paper, inspired by \cite{Bai2018,gower2021adaptive,zhang2020greedy}, where the capped threshold is introduced to ensure that relatively large component values are selected, we first propose the distance-residual capped nonlinear Kaczmarz (DR-CNK) method for solving the nonlinear problem (\ref{Sec11}). Here the capped threshold is mainly determined by the maximal distance rule which is discussed in \cite{zhang2022greedy} in detail. Similarly, the residual-distance capped nonlinear Kaczmarz (RD-CNK) method is also presented, whose capped threshold is mainly determined by the maximal residual rule \cite{zhang2022greedy}. We prove that the two new methods converge linearly in expectation with convergence factors being strictly smaller than that of the NRK and NURK methods presented in \cite{wang2022nonlinear}. Furthermore, two block versions, i.e., the multiple samples-based methods, are also proposed to accelerate our new methods.

The rest of this paper is organized as follows. Section \ref{sec:preliminaries} provides some preliminaries.  In Section \ref{sec: Single-methods}, we propose the DR-CNK and RD-CNK methods 
and their convergence analysis are given in Section \ref{sec:theory_Single-methods}. The block versions of the new methods are presented in Section \ref{sec: mutiple-methods} and the relevant convergence theorems are proved in Section \ref{sec:theory_mutiple-methods}. Experimental results are shown in Section \ref{sec:experiments}. Finally, we conclude the paper with some remarks.

\section{Preliminaries}
\label{sec:preliminaries}

 For a matrix $A=(A_{(i, j)})\in \mathbb{R}^{m\times n}$, $\sigma_{\max}(A)$, $\|A\|_2$, $\|A\|_F$, $A^{\dag}$, and $A_{\tau}$ denote its largest singular value, spectral norm, Frobenius norm,  Moore-Penrose pseudoinverse, and the restriction onto the row indices in the set $\tau$, respectively. We use $\lvert \tau \rvert$, $ \mathbb{E}^{k}$, and $ \mathbb{E}$ to denote the number of elements of a set $\mathcal{\tau}$, the conditional expectation conditioned on the first $k$ iterations, and the full expected value, respectively.  In addition, we let $[m]:=\{1, \cdots, m\}$ for an integer $m\geq 1$ and define $$ h_2(A)=\mathop{\text{inf}}\limits_{x\neq 0}\frac{\|Ax\|_2}{\|x\|_2},$$ and $$f^{\prime}(x)=[\nabla f_1(x), \cdots, \nabla f_m(x)]^T\in\mathbb{R}^{m\times n},$$ which is the Jacobian matrix of $f$ at $x$.

In addition, the following facts are necessary throughout the paper.
\begin{definition}[\cite{2017Kaczmarz}]\label{definition}
If for $i \in [ m ]$ and $\forall  x_{1}, x_{2} \in \mathbb{R}^{n}$, there exists $\eta_{i} \in[0, \eta)$ satisfying $\eta=\max\limits _{i}
\eta_{i}<\frac{1}{2}$ such that
\begin{align}
 \lvert f_{i}\left({x}_{1}\right)-f_{i}\left({x}_{2}\right)-\nabla f_{i}\left({x}_{1}\right)^T\left({x}_{1}-{x}_{2}\right) \rvert
\leq \eta_{i} \lvert f_{i}\left({x}_{1}\right)-f_{i}\left({x}_{2}\right) \rvert,\notag
\end{align}
then the function $f: \mathbb{R}^{n} \rightarrow \mathbb{R}^{m}$ is referred to satisfy the local tangential cone condition.
\end{definition}

\begin{lemma}[\cite{wang2022nonlinear}]\label{lemma1}
If the function $f$ satisfies the local tangential cone condition, then for $i_k \in [ m ]$, $\forall x_{1}, x_{2} \in \mathbb{R}^{n}$ and the updating formula (\ref{Sec1.4}), we have
\begin{align}
\left\|x_{k+1}-x_{\star}\right\|^{2}_2 \leq \left\|x_{k}-x_{\star}\right\|^{2}_2-\left(1-2 \eta_{i_k}\right)
\frac{\left\lvert f_{i_k}(x_k)\right\rvert^{2}}{\left\|\nabla f_{i_k}\left(x_{k}\right)\right\|_2^{2}}.\notag 
\end{align}
\end{lemma}

\begin{lemma}[\cite{zhang2022greedy}]  \label{lemma3}
If the function $f $ satisfies the local tangential cone condition, then for $\forall x_{1}, x_{2} \in \mathbb{R}^{n}$ and an index subset $\tau\subseteq[m]$, we have
\begin{align}
\left\|f_{\tau}\left({x}_{1}\right)-f_{\tau}\left({x}_{2}\right)\right\|^2_2
\geq \frac{1}{1+\eta^2}\left\|f^{\prime}_{\tau}\left(x_{1}\right)\left(x_{1}-x_{2}\right)\right\|^2_2.\notag
\end{align}
\end{lemma}

\begin{lemma}[\cite{zhang2022greedy}]  \label{sec:theory-mutiple-lemma1}
If the function $f $ satisfies the local tangential cone condition and a vector $x_\star\in \mathbb{R}^{n} $ satisfies $f(x_\star)=0$, then from the block iteration formula $x_{k+1}=x_{k}-(f_{\tau_k}^{\prime}(x_k))^{\dagger}f_{\tau_k}(x_k)$ with $\tau_k\subseteq[m]$, we have
\begin{align}
\left\|x_{k+1}-x_{\star}\right\|^{2}_2
\leq
\left\|x_{k}-x_{\star}\right\|^{2}_2-  \left(h_2^2((f_{\tau_k}^{\prime}(x_k))^{\dagger})-2\eta\sigma^2_{\max}\left((f_{\tau_k}^{\prime}(x_k))^{\dagger}\right)\right)\left\|  f_{\tau_k}(x_k)\right\|^{2}_2. \notag
\end{align}
\end{lemma}

\section{Single sample-based capped methods}
\label{sec: Single-methods}

Intuitively, at the $k$th iteration, to make $f(x)\rightarrow 0$ as fast as possible, we might want the larger entries in the vector $f(x)$ to be preferentially annihilated as much as possible. In this way, the corresponding nonlinear Kaczmarz method should converge quickly. With this in mind, we construct the DR-CNK method listed in Algorithm \ref{The DR-CNK method}. Specifically, it first determines the index subset $\mathcal{U}_k$ in (\ref{AL_DR-CNK_Uk}) by using the combination of the maximum and average distances, and then samples an index from the subset $\mathcal{U}_k$ with probability that is proportional to the corresponding residual. This means that the final selected iteration index $i_k$ in the DR-CNK method is jointly determined by two greedy rules, namely, the distance rule for determining the index set $\mathcal{U}_k$ and the residual rule for extracting the index $i_k$ from the set $\mathcal{U}_k$. 
Conversely, we can also use the maximum residual rule to construct an index subset $\mathcal{I}_k$, and then select an iteration index $i_k$ from the set with the probability criterion determined by the distance, thus obtaining the RD-CNK method shown in Algorithm \ref{The RD-CNK method}.

\begin{algorithm}
\caption{The DR-CNK method}
\label{The DR-CNK method}
\begin{algorithmic}[1]
\Require The initial estimate $x_0\in \mathbb{R}^n$.
\For{$k=0, 1, 2, \cdots $ until convergence,}
        \State Compute
\begin{align}
\varepsilon_k=\frac{1}{2}\left(\frac{1}{\|f(x_k)\|_2^2}\max\limits_{i\in[m]}\frac{\left\lvert f_{i}(x_k)\right\rvert^{2}}{\left\|\nabla f_{i}\left(x_{k}\right)\right\|_2^{2}}+\frac{1}{\|f^{\prime}(x_k)\|_F^2}\right). \label{AL_DR-CNK_vareps}
\end{align}
        \State Determine the index subset
\begin{align}
\mathcal{U}_k=\left\{i\vert \left\lvert f_{i}(x_k)\right\rvert^{2}\geq\varepsilon_k\|f(x_k)\|_2^2\left\|\nabla f_{i}\left(x_{k}\right)\right\|_2^{2}\right\} . \label{AL_DR-CNK_Uk}
\end{align}
        \State Compute the $i$th entry $\tilde{r}_k^{(i)}$ of the vector $\tilde{r}_k$ according to
  $$
\tilde{r}_{k}^{(i)}=\left\{\begin{array}{ll}{f_i(x_k),} & {\text { if } i \in \mathcal{U}_{k}}, \\ {0,} & {\text { otherwise. }}\end{array}\right.
$$
        \State Select $i_{k} \in \mathcal{U}_{k}$ with probability $\operatorname{Pr}\left(\mathrm{row}=i_{k}\right)=\frac{\lvert\tilde{r}_{k}^{\left(i_{k}\right)}\rvert^{2}}{\left\|\tilde{r}_{k}\right\|_{2}^{2}}$.
        \State Update $x_{k+1}=x_{k}-\frac{f_{i_k}(x_k)}{\|\nabla f_{i_k}(x_k)\|^2_2}\nabla f_{i_k}(x_k)$.
\EndFor
\end{algorithmic}
\end{algorithm}
\begin{algorithm}
\caption{The RD-CNK method}
\label{The RD-CNK method}
\begin{algorithmic}[1]
\Require The initial estimate $x_0\in \mathbb{R}^n$.
\For{$k=0, 1, 2, \cdots $ until convergence,}
\State{Compute
\begin{align}
\delta_k=\frac{1}{2}\left(\frac{\max\limits_{i\in[m]}\left\lvert f_{i}(x_k)\right\rvert^{2}}{\|f(x_k)\|_2^2} +\frac{1}{m}\right) . \label{AL_RD-CNK_delta}
\end{align}
}
\State{Determine the index subset
\begin{align}
\mathcal{I}_k=\left\{i\vert \left\lvert f_{i}(x_k)\right\rvert^{2}\geq\delta_k\|f(x_k)\|_2^2 \right\} . \label{AL_RD-CNK_Ik}
\end{align}
}
\State{Compute the $i$th entry $\tilde{d}_k^{(i)}$ of the vector $\tilde{d}_k$ according to
  $$
\tilde{d}_{k}^{(i)}=\left\{\begin{array}{ll}{\frac{ f_{i}(x_k) }{\left\|\nabla f_{i}\left(x_{k}\right)\right\|_2},} & {\text { if } i \in \mathcal{I}_{k}}, \\ {0,} & {\text { otherwise. }}\end{array}\right.
$$}
\State{Select $i_{k} \in \mathcal{I}_{k}$ with probability $\operatorname{Pr}\left(\mathrm{row}=i_{k}\right)=\frac{\lvert\tilde{d}_{k}^{\left(i_{k}\right)}\rvert^{2}}{\left\|\tilde{d}_{k}\right\|_{2}^{2}}$.}
\State{Update $x_{k+1}=x_{k}-\frac{f_{i_k}(x_k)}{\|\nabla f_{i_k}(x_k)\|^2_2}\nabla f_{i_k}(x_k)$. }
\EndFor
\end{algorithmic}
\end{algorithm}

\begin{remark}
\label{Al_rmk1}
The DR-CNK method is well defined as the index set $\mathcal{U}_k$ in (\ref{AL_DR-CNK_Uk}) is always nonempty.
This is because
$$
 \max\limits_{i\in[m]}\frac{\left\lvert f_{i}(x_k)\right\rvert^{2}}{\left\|\nabla f_{i}\left(x_{k}\right)\right\|_2^{2}}\geq \sum \limits _{i=1}^{m}\frac{ \left\|\nabla f_{i}\left(x_{k}\right)\right\|_2^{2}}{\|f^{\prime}(x_k)\|_F^2}\frac{\left\lvert f_{i}(x_k)\right\rvert^{2}}{\left\|\nabla f_{i}\left(x_{k}\right)\right\|_2^{2}}=\frac{\|f(x_k)\|_2^2}{\|f^{\prime}(x_k)\|_F^2}$$
and then
$$
\frac{\left\lvert f_{i_k}(x_k)\right\rvert^{2}}{\left\|\nabla f_{i_k}\left(x_{k}\right)\right\|_2^{2}} = \max\limits_{i\in[m]}\frac{\left\lvert f_{i}(x_k)\right\rvert^{2}}{\left\|\nabla f_{i}\left(x_{k}\right)\right\|_2^{2}}  \geq\frac{1}{2}\left( \max\limits_{i\in[m]}\frac{\left\lvert f_{i}(x_k)\right\rvert^{2}}{\left\|\nabla f_{i}\left(x_{k}\right)\right\|_2^{2}}+\frac{\|f(x_k)\|_2^2}{\|f^{\prime}(x_k)\|_F^2}\right)
$$
imply $i_k\in\mathcal{U}_{k}.$

Similarly, we can show that the index set $\mathcal{I}_{k}$ in Algorithm \ref{The RD-CNK method} is also nonempty, that is, the RD-CNK method is also well defined.

In addition, if $f(x)=Ax-b$, the DR-CNK and RD-CNK methods recover the greedy randomized Kaczmarz method \cite{Bai2018} and the greedy randomized Motzkin-Kaczmarz method \cite{zhang2020greedy}, respectively.
\end{remark}

\begin{remark}
\label{Al_rmk2}
As in \cite{Bai2018r,gower2021adaptive}, a relaxation parameter $\theta \in [0, 1]$ can be introduced into the quantities $\varepsilon_{k}$ in (\ref{AL_DR-CNK_vareps}) and $\delta_{k}$ in (\ref{AL_RD-CNK_delta}), thus obtaining
\begin{align}
\varepsilon_k=  \theta\frac{1}{\|f(x_k)\|_2^2}\max\limits_{i\in[m]}\frac{\left\lvert f_{i}(x_k)\right\rvert^{2}}{\left\|\nabla f_{i}\left(x_{k}\right)\right\|_2^{2}}+(1-\theta)\frac{1}{\|f^{\prime}(x_k)\|_F^2}  \notag
\end{align}
and
\begin{align}
\delta_k=\theta\frac{\max\limits_{i\in[m]}\left\lvert f_{i}(x_k)\right\rvert^{2}}{\|f(x_k)\|_2^2} +(1-\theta)\frac{1}{m}. \notag
\end{align}
Then, the corresponding relaxed greedy capped nonlinear Kaczmarz methods can be obtained.
\end{remark}

\section{Convergence analysis}
\label{sec:theory_Single-methods}

Below, we present the convergence guarantees for the DR-CNK and RD-CNK methods.

\begin{theorem}
\label{th:DR-CNK}
If the nonlinear function $f $ satisfies the local tangential cone condition given in Definition \ref{definition}, $\eta=\max\limits _{i\in[m]}
\eta_{i}<\frac{1}{2}$, and $f(x_{\star})=0$, then the iterations of the DR-CNK method in Algorithm \ref{The DR-CNK method} satisfy
\begin{align}
\mathbb{E}\left[\left\| {x}_{k+1}- {x}_{\star}\right\|^{2}_2\right] \leq\left(1 - \frac{ 1-2 \eta }{1+\eta^2} \varepsilon_k h_2^2\left( f^{\prime}\left(x_{k}\right)\right )     \right)
 \mathbb{E}\left[\left\|x_{k}-{x}_{\star}\right\|^{2}_2\right].\label{th:DR-CNK_0}
 \end{align}
\end{theorem}

\begin{proof}
From Lemma \ref{lemma1}, we have
\begin{align}
\left\|x_{k+1}-x_{\star}\right\|^{2}_2
&\leq \left\|x_{k}-x_{\star}\right\|^{2}_2-\left(1-2 \eta_{i_k}\right)
\frac{\left\lvert f_{i_k}(x_k)\right\rvert^{2}}{\left\|\nabla f_{i_k}\left(x_{k}\right)\right\|_2^{2}}. \notag
\end{align}
Taking expectation of both sides conditioned on $x_{k}$ gives
\begin{align}
\mathbb{E}^{k}\left[\left\| {x}_{k+1}- {x}_{\star}\right\|^{2}_2\right]
&\leq
\left\|x_{k}-x_{\star}\right\|^{2}_2 -\mathbb{E}^{k}\left[\left(1-2 \eta_{i_k}\right)
\frac{\left\lvert f_{i_k}(x_k)\right\rvert^{2}}{\left\|\nabla f_{i_k}\left(x_{k}\right)\right\|_2^{2}}\right] \notag
\\
&\leq
\left\|x_{k}-x_{\star}\right\|^{2}_2 -\left(1-2 \eta \right)\mathbb{E}^{k}\left[
\frac{\left\lvert f_{i_k}(x_k)\right\rvert^{2}}{\left\|\nabla f_{i_k}\left(x_{k}\right)\right\|_2^{2}}\right] \notag
\\
&=
\left\|x_{k}-x_{\star}\right\|^{2}_2 -\left(1-2 \eta \right)   \sum\limits_{i_k\in\mathcal{U}_k} \frac{\lvert\tilde{r}_{k}^{\left(i_{k}\right)}\rvert^{2}}{\left\|\tilde{r}_{k}\right\|_{2}^{2}}
\frac{\left\lvert f_{i_k}(x_k)\right\rvert^{2}}{\left\|\nabla f_{i_k}\left(x_{k}\right)\right\|_2^{2}}, \notag
\end{align}
which together with the definitions of $\varepsilon_k$ in (\ref{AL_DR-CNK_vareps}) and the index subset $\mathcal{U}_k$ in (\ref{AL_DR-CNK_Uk}) leads to
\begin{align}
\mathbb{E}^{k}\left[\left\| {x}_{k+1}- {x}_{\star}\right\|^{2}_2\right]
&\leq
\left\|x_{k}-x_{\star}\right\|^{2}_2 -\left(1-2 \eta \right)   \sum\limits_{i_k\in\mathcal{U}_k} \frac{\lvert\tilde{r}_{k}^{\left(i_{k}\right)}\rvert^{2}}{\left\|\tilde{r}_{k}\right\|_{2}^{2}}
\varepsilon_k\|f(x_k)\|_2^2 \notag
\\
&=
\left\|x_{k}-x_{\star}\right\|^{2}_2 -\left(1-2 \eta \right)
\varepsilon_k\|f(x_k)-f(x_{\star})\|_2^2. \notag
\end{align}
Further, considering Lemma \ref{lemma3}, we obtain
\begin{align}
\mathbb{E}^{k}\left[\left\| {x}_{k+1}- {x}_{\star}\right\|^{2}_2\right]
&\leq
\left\|x_{k}-x_{\star}\right\|^{2}_2 -\left(1-2 \eta \right)
\varepsilon_k\frac{1}{1+\eta^2}\left\|f^{\prime} \left(x_{k}\right)\left(x_{k}-x_{\star}\right)\right\|^2_2 \notag
\\
&\leq
\left\|x_{k}-x_{\star}\right\|^{2}_2 -\frac{ 1-2 \eta }{1+\eta^2}\varepsilon_k h_2^2(f^{\prime} \left(x_{k}\right))\left\| x_{k}-x_{\star} \right\|^2_2  \notag
\\
&=
\left(1 -\frac{ 1-2 \eta }{1+\eta^2}\varepsilon_k h_2^2(f^{\prime} \left(x_{k}\right))\right)\left\| x_{k}-x_{\star} \right\|^2_2.  \notag
\end{align}
So, the desired result (\ref{th:DR-CNK_0}) can be deduced by taking expectation on both sides and using the tower rule of expectation.
\end{proof}

\begin{remark}\label{rek:th_DR-CNK_1}
According to $\varepsilon_k $ in (\ref{AL_DR-CNK_vareps}), we have
 \begin{align}
\varepsilon_k \|f^{\prime} \left(x_{k}\right)\|_F^2
&=\frac{1}{2}\left(\frac{\|f^{\prime}(x_k)\|_F^2}{\|f(x_k)\|_2^2}\max\limits_{i\in[m]}\frac{\left\lvert f_{i}(x_k)\right\rvert^{2}}{\left\|\nabla f_{i}\left(x_{k}\right)\right\|_2^{2}}+1\right)\notag
\\
&=
\frac{1}{2}\left(\max\limits_{i\in[m]}\frac{\left\lvert f_{i}(x_k)\right\rvert^{2}}{\left\|\nabla f_{i}\left(x_{k}\right)\right\|_2^{2}}/\frac{\|f(x_k)\|_2^2}{\|f^{\prime}(x_k)\|_F^2}+1\right) \notag
\\
&=
\frac{1}{2}\left(\max\limits_{i\in[m]}\frac{\left\lvert f_{i}(x_k)\right\rvert^{2}}{\left\|\nabla f_{i}\left(x_{k}\right)\right\|_2^{2}}/\sum\limits_{i \in[m]} \frac{\left\|\nabla f_{i}\left(x_{k}\right)\right\|_2^{2}}{\|f^{\prime}(x_k)\|_F^2}\frac{\lvert f_i(x_k)\rvert^2}{\left\|\nabla f_{i}\left(x_{k}\right)\right\|_2^{2}}+1\right) \notag
\\
&\geq 1.\notag
 \end{align}
That is, $$\varepsilon_k\geq\frac{1}{ \|f^{\prime} \left(x_{k}\right)\|_F^2}.$$
Then, it holds that
\begin{align}
 \frac{ 1-2 \eta }{1+\eta^2} \varepsilon_k h_2^2\left( f^{\prime}\left(x_{k}\right)\right )>\frac{1-2 \eta}{(1+\eta)^{2}   }
 \frac{h_2^2(f^{\prime}(x_{k}))}{ \|f^{\prime}(x_{k}) \|_{F}^{2}m   },\notag
 \end{align}
which implies
\begin{align}
\rho_{\text{DR-CNK}}=1-\frac{ 1-2 \eta }{1+\eta^2} \varepsilon_k h_2^2\left( f^{\prime}\left(x_{k}\right)\right )<1-\frac{1-2 \eta}{(1+\eta)^{2}   }
 \frac{h_2^2(f^{\prime}(x_{k}))}{ \|f^{\prime}(x_{k}) \|_{F}^{2}m   }=\rho_{\text{NRK}}=\rho_{\text{NURK}}.\notag
 \end{align}
Here $\rho_{\text{NRK}} $ and $ \rho_{\text{NURK}}$ are respectively the convergence factors of the NRK and NURK methods provided in \cite{wang2022nonlinear}.
So, we can conclude that the convergence factor of the DR-CNK method is strictly smaller than that of the NRK and NURK methods.
\end{remark}

\begin{theorem}
\label{th:RD-CNK}
If the nonlinear function $f $ satisfies the local tangential cone condition given in Definition \ref{definition}, $\eta=\max\limits _{i\in[m]}
\eta_{i}<\frac{1}{2}$, and $f(x_{\star})=0$, then the iterations of the RD-CNK method in Algorithm \ref{The RD-CNK method} satisfy
\begin{align}
\mathbb{E}\left[\left\| {x}_{k+1}- {x}_{\star}\right\|^{2}_2\right] \leq\left(1 - \frac{ 1-2 \eta }{1+\eta^2} \frac{ \delta_k  h_2^2\left( f^{\prime}\left(x_{k}\right)\right )  }{\max\limits_{i\in[m]}\left\|\nabla f_{i}\left(x_{k}\right)\right\|_2^{2}}   \right)
 \mathbb{E}\left[\left\|x_{k}-{x}_{\star}\right\|^{2}_2\right].\label{th:RD-CNK_0}
 \end{align}
\end{theorem}

\begin{proof}
Following an analogous proof process to the DR-CNK method, we can get the inequality
\begin{align}
\mathbb{E}^{k}\left[\left\| {x}_{k+1}- {x}_{\star}\right\|^{2}_2\right]
&\leq
\left\|x_{k}-x_{\star}\right\|^{2}_2 -\left(1-2 \eta \right)\mathbb{E}^{k}\left[
\frac{\left\lvert f_{i_k}(x_k)\right\rvert^{2}}{\left\|\nabla f_{i_k}\left(x_{k}\right)\right\|_2^{2}}\right] \notag
\\
&\leq
\left\|x_{k}-x_{\star}\right\|^{2}_2 -\left(1-2 \eta \right) \frac{ 1 }{\max\limits_{i\in[m]}\left\|\nabla f_{i}\left(x_{k}\right)\right\|_2^{2}}\mathbb{E}^{k}\left[
\left\lvert f_{i_k}(x_k)\right\rvert^{2}\right] \notag
\\
&\leq
\left\|x_{k}-x_{\star}\right\|^{2}_2 -\frac{  1-2 \eta   }{\max\limits_{i\in[m]}\left\|\nabla f_{i}\left(x_{k}\right)\right\|_2^{2}} \sum\limits_{i_k\in\mathcal{I}_k} \frac{\lvert\tilde{d}_{k}^{\left(i_{k}\right)}\rvert^{2}}{\left\|\tilde{d}_{k}\right\|_{2}^{2}}
\left\lvert f_{i_k}(x_k)\right\rvert^{2}, \notag
\end{align}
and from the definitions of $\delta_k$ in (\ref{AL_RD-CNK_delta}) and the index subset $\mathcal{I}_k$ in (\ref{AL_RD-CNK_Ik}), we can further obtain
\begin{align}
\mathbb{E}^{k}\left[\left\| {x}_{k+1}- {x}_{\star}\right\|^{2}_2\right]
&\leq
\left\|x_{k}-x_{\star}\right\|^{2}_2 - \frac{ 1-2 \eta  }{\max\limits_{i\in[m]}\left\|\nabla f_{i}\left(x_{k}\right)\right\|_2^{2}} \sum\limits_{i_k\in\mathcal{I}_k} \frac{\lvert\tilde{d}_{k}^{\left(i_{k}\right)}\rvert^{2}}{\left\|\tilde{d}_{k}\right\|_{2}^{2}}
\delta_k\|f(x_k)\|_2^2\notag
\\
&=
\left\|x_{k}-x_{\star}\right\|^{2}_2 -\frac{  1-2 \eta  }{\max\limits_{i\in[m]}\left\|\nabla f_{i}\left(x_{k}\right)\right\|_2^{2}}
\delta_k\|f(x_k)-f(x_{\star})\|_2^2,\notag
\end{align}
which together with Lemma \ref{lemma3} leads to
\begin{align}
\mathbb{E}^{k}\left[\left\| {x}_{k+1}- {x}_{\star}\right\|^{2}_2\right]
&\leq
\left\|x_{k}-x_{\star}\right\|^{2}_2 - \frac{\left(1-2 \eta \right) \delta_k }{\max\limits_{i\in[m]}\left\|\nabla f_{i}\left(x_{k}\right)\right\|_2^{2}}
\frac{1}{1+\eta^2}\left\|f^{\prime} \left(x_{k}\right)\left(x_{k}-x_{\star}\right)\right\|^2_2\notag
\\
&\leq
\left\|x_{k}-x_{\star}\right\|^{2}_2 -\frac{ 1-2 \eta }{1+\eta^2} \frac{ \delta_k }{\max\limits_{i\in[m]}\left\|\nabla f_{i}\left(x_{k}\right)\right\|_2^{2}}h_2^2(f^{\prime} \left(x_{k}\right))
\left\| x_{k}-x_{\star} \right\|^2_2\notag
\\
&=
\left(1 -\frac{ 1-2 \eta }{1+\eta^2} \frac{ \delta_k h_2^2(f^{\prime} \left(x_{k}\right)) }{\max\limits_{i\in[m]}\left\|\nabla f_{i}\left(x_{k}\right)\right\|_2^{2}}\right)
\left\| x_{k}-x_{\star} \right\|^2_2.\notag
\end{align}
Thus, by taking the full expectation on both sides, we get the desired result (\ref{th:RD-CNK_0}).
\end{proof}

\begin{remark}\label{rek:th_RD-CNK_1}
Since $\max\limits_{i\in[m]}\left\|\nabla f_{i}\left(x_{k}\right)\right\|_2^{2}\leq \|f^{\prime} \left(x_{k}\right)\|^2_F$ and $
 \delta_k=\frac{1}{2}\left(\frac{\max\limits_{i\in[m]}\left\lvert f_{i}(x_k)\right\rvert^{2}}{\|f(x_k)\|_2^2} +\frac{1}{m}\right) \geq\frac{1}{m}$, we have
 $$ \rho_{\text{RD-CNK}}=1 -\frac{ 1-2 \eta }{1+\eta^2} \frac{ \delta_k h_2^2(f^{\prime} \left(x_{k}\right))}{\max\limits_{i\in[m]}\left\|\nabla f_{i}\left(x_{k}\right)\right\|_2^{2}}  <1-\frac{1-2 \eta}{(1+\eta)^{2}   }
 \frac{h_2^2(f^{\prime}(x_{k}))}{ \|f^{\prime}(x_{k}) \|_{F}^{2}m   }=\rho_{\text{NRK}}=\rho_{\text{NURK}},$$
which means that the convergence factor of the RD-CNK method is strictly smaller than that of the NRK and NURK methods presented in \cite{wang2022nonlinear}.
\end{remark}

\section{Multiple samples-based capped methods}
\label{sec: mutiple-methods}
After determining the index subsets $\mathcal{U}_k$ in Algorithm \ref{The DR-CNK method} and $\mathcal{I}_k$ in Algorithm \ref{The RD-CNK method}, we can directly project the current iteration $x_k$ onto the solution space of these subsets leading to the corresponding block methods, i.e., the distance-based block capped nonlinear Kaczmarz (DB-CNK) method listed in Algorithm \ref{The DB-CNK method} and the residual-based block capped nonlinear Kaczmarz (RB-CNK) method shown in Algorithm \ref{The RB-CNK method}.

\begin{algorithm}
\caption{The DB-CNK method}
\label{The DB-CNK method}
\begin{algorithmic}[1]
\Require The initial estimate $x_0\in \mathbb{R}^n$.
\For{$k=0, 1, 2, \cdots $ until convergence,}
\State{Compute $\varepsilon_k=\frac{1}{2}\left(\frac{1}{\|f(x_k)\|_2^2}\max\limits_{i\in[m]}\frac{\left\lvert f_{i}(x_k)\right\rvert^{2}}{\left\|\nabla f_{i}\left(x_{k}\right)\right\|_2^{2}}+\frac{1}{\|f^{\prime}(x_k)\|_F^2}\right) $.
}
\State{Determine the index subset
$$
\mathcal{U}_k=\left\{i\vert \left\lvert f_{i}(x_k)\right\rvert^{2}\geq\varepsilon_k\|f(x_k)\|_2^2\left\|\nabla f_{i}\left(x_{k}\right)\right\|_2^{2}\right\} .
$$
}
\State{Update $x_{k+1}=x_{k}-(f_{\mathcal{U}_k}^{\prime}(x_k))^{\dagger}f_{\mathcal{U}_k}(x_k)$. }
\EndFor
\end{algorithmic}
\end{algorithm}

\begin{algorithm}
\caption{The RB-CNK method}
\label{The RB-CNK method}
\begin{algorithmic}[1]
\Require The initial estimate $x_0\in \mathbb{R}^n$.
\For{$k=0, 1, 2, \cdots $ until convergence,}
\State{Compute $
\delta_k=\frac{1}{2}\left(\frac{\max\limits_{i\in[m]}\left\lvert f_{i}(x_k)\right\rvert^{2}}{\|f(x_k)\|_2^2} +\frac{1}{m}\right)$.
}
\State{Determine the index subset
$$
\mathcal{I}_k=\left\{i\vert \left\lvert f_{i}(x_k)\right\rvert^{2}\geq\delta_k\|f(x_k)\|_2^2 \right\}.
$$
}
\State{Update $x_{k+1}=x_{k}-(f_{\mathcal{I}_k}^{\prime}(x_k))^{\dagger}f_{\mathcal{I}_k}(x_k)$. }
\EndFor
\end{algorithmic}
\end{algorithm}

\begin{remark}
\label{Al34_rmk}
Since the updating index $i_k$ in the DR-CNK method belongs to the subset $\mathcal{U}_k$ which is directly used in the DB-CNK method, we can deduce that the DB-CNK method must converge at least as fast as the DR-CNK method. We can also obtain a similar relationship between the RB-CNK and RD-CNK methods as the index $i_k$ used in the latter also belongs to the subset $\mathcal{I}_k$ which is directly used in the former.

In addition, a parameter $\xi\in (0, 1]$ can be introduced into $\varepsilon_k$ and $\delta_k$ and then obtaining $$\varepsilon_k= \xi\frac{1}{\|f(x_k)\|_2^2}\max\limits_{i\in[m]}\frac{\left\lvert f_{i}(x_k)\right\rvert^{2}}{\left\|\nabla f_{i}\left(x_{k}\right)\right\|_2^{2}} \quad \text{and} \quad
\delta_k=\xi\frac{\max\limits_{i\in[m]}\left\lvert f_{i}(x_k)\right\rvert^{2}}{\|f(x_k)\|_2^2} $$ as discussed in \cite{niu2020greedy,zhang2020greedy}. Further, if $\lvert \mathcal{U}_k \rvert=\lvert \mathcal{I}_k \rvert=1$, the DB-CNK and RB-CNK methods will respectively reduce to the MD-NK and MR-NK methods presented in \cite{zhang2022greedy}.
\end{remark}

\section{Convergence analysis}
\label{sec:theory_mutiple-methods}

In this section, we establish the convergence theorems of the DB-CNK and RB-CNK methods.

\begin{theorem}
\label{th:DB-CNK}
If the nonlinear function $f $ satisfies the local tangential cone condition given in Definition \ref{definition}, $\eta=\max\limits _{i\in[m]}
\eta_{i}<\frac{1}{2}$, $f(x_{\star})=0$ and $
\alpha=h_2^2((f_{\mathcal{U} }^{\prime}(x_k))^{\dagger})-2\eta\sigma^2_{\max}\left((f_{\mathcal{U} }^{\prime}(x_k))^{\dagger}\right)>0,$
where
\begin{align}
h_2^2((f_{\mathcal{U}}^{\prime}(x_k))^{\dagger})= \min \limits _{\mathcal{U}_{k}} h_2^2((f_{\mathcal{U}_{k}}^{\prime}(x_k))^{\dagger}) \textrm{ and } \sigma^2_{\max}\left((f_{\mathcal{U} }^{\prime}(x_k))^{\dagger}\right)= \max \limits _{\mathcal{U}_{k}}\sigma^2_{\max}\left((f_{\mathcal{U}_{k}}^{\prime}(x_k))^{\dagger}\right), \notag
\end{align}
then the iterations of the DB-CNK method in Algorithm \ref{The DB-CNK method} satisfy
\begin{align}
\left\| {x}_{k+1}- {x}_{\star}\right\|^{2}_2 \leq
\left(1-\alpha \min \limits _{i\in[m]} \left\|\nabla f_{i}\left(x_{k}\right)\right\|_2^{2} \lvert\mathcal{U}_{k}\rvert \varepsilon_k\frac{1}{1+\eta^2}h_2^2(f^{\prime} \left(x_{k}\right))\right)
\left\|x_{k}-{x}_{\star}\right\|^{2}_2.\label{th:DB-CNK_0}
 \end{align}
\end{theorem}

\begin{proof}
According to Lemma \ref{sec:theory-mutiple-lemma1}, the definition of $\alpha$, and Algorithm \ref{The DB-CNK method}, we have
\begin{align}
\left\|x_{k+1}-x_{\star}\right\|^{2}_2
&\leq
\left\|x_{k}-x_{\star}\right\|^{2}_2-  \left(h_2^2((f_{\mathcal{U}_{k}}^{\prime}(x_k))^{\dagger})-2\eta\sigma^2_{\max}\left((f_{\mathcal{U}_{k}}^{\prime}(x_k))^{\dagger}\right)\right)\left\|  f_{\mathcal{U}_{k}}(x_k)\right\|^{2}_2  \notag
\\
&\leq
\left\|x_{k}-x_{\star}\right\|^{2}_2- \alpha\left\|  f_{\mathcal{U}_{k}}(x_k)\right\|^{2}_2 \notag
\\
&=
\left\|x_{k}-x_{\star}\right\|^{2}_2- \alpha \sum\limits_{j\in\mathcal{U}_{k} } \frac{\left\lvert  f_{j}(x_k)\right\rvert^{2} }{\left\|\nabla f_{j}\left(x_{k}\right)\right\|_2^{2}}\left\|\nabla f_{j}\left(x_{k}\right)\right\|_2^{2}\notag
\\
&\leq
\left\|x_{k}-x_{\star}\right\|^{2}_2-\alpha  \min \limits _{i\in[m]} \left\|\nabla f_{i}\left(x_{k}\right)\right\|_2^{2} \sum\limits_{j\in\mathcal{U}_{k} } \frac{\left\lvert  f_{j}(x_k)\right\rvert^{2} }{\left\|\nabla f_{j}\left(x_{k}\right)\right\|_2^{2}}  \notag
\\
&\leq
\left\|x_{k}-x_{\star}\right\|^{2}_2-\alpha \min \limits _{i\in[m]} \left\|\nabla f_{i}\left(x_{k}\right)\right\|_2^{2} \lvert\mathcal{U}_{k}\rvert \varepsilon_k\|f(x_k)\|_2^2, \notag
\end{align}
which together with Lemma \ref{lemma3} yields
\begin{align}
\left\|x_{k+1}-x_{\star}\right\|^{2}_2
&\leq
\left\|x_{k}-x_{\star}\right\|^{2}_2-\alpha \min \limits _{i\in[m]} \left\|\nabla f_{i}\left(x_{k}\right)\right\|_2^{2} \lvert\mathcal{U}_{k}\rvert \varepsilon_k\frac{1}{1+\eta^2}\left\|f^{\prime} \left(x_{k}\right)\left(x_{k}-x_{\star}\right)\right\|^2_2 \notag
\\
&\leq
\left(1-\alpha \min \limits _{i\in[m]} \left\|\nabla f_{i}\left(x_{k}\right)\right\|_2^{2} \lvert\mathcal{U}_{k}\rvert \varepsilon_k\frac{1}{1+\eta^2}h_2^2(f^{\prime} \left(x_{k}\right))\right)\left\|x_{k}-x_{\star}\right\|^2_2. \notag
\end{align}
So, the desired result (\ref{th:DB-CNK_0}) is obtained.
\end{proof}

\begin{remark}
\label{rek_th_DB-CNK-1}
Since $0\leq  \left\| {x}_{k+1}- {x}_{\star}\right\|^{2}_2 $ and $\alpha \min \limits _{i\in[m]} \left\|\nabla f_{i}\left(x_{k}\right)\right\|_2^{2} \lvert\mathcal{U}_{k}\rvert \varepsilon_k\frac{h_2^2(f^{\prime} \left(x_{k}\right))}{1+\eta^2}>0$, we have
\begin{align}
0
\leq
\left\|x_{k}-x_{\star}\right\|^2_2-\alpha \min \limits _{i\in[m]} \left\|\nabla f_{i}\left(x_{k}\right)\right\|_2^{2} \lvert\mathcal{U}_{k}\rvert \varepsilon_k\frac{h_2^2(f^{\prime} \left(x_{k}\right))}{1+\eta^2} \left\|x_{k}-x_{\star}\right\|^2_2
 <\left\|x_{k}-x_{\star}\right\|^2_2,\notag
 \end{align}
which implies that the convergence factor of the DB-CNK method is smaller than 1. Similarly, we can get that the convergence factor of the RB-CNK method presented in Theorem \ref{th:RB-CNK} is also smaller than 1.
\end{remark}

\begin{theorem}
\label{th:RB-CNK}
If the nonlinear function $f $ satisfies the local tangential cone condition given in Definition \ref{definition}, $\eta=\max\limits _{i\in[m]}
\eta_{i}<\frac{1}{2}$, $f(x_{\star})=0$ and $
\beta=h_2^2((f_{\mathcal{I} }^{\prime}(x_k))^{\dagger})-2\eta\sigma^2_{\max}\left((f_{\mathcal{I} }^{\prime}(x_k))^{\dagger}\right)>0,$
where
\begin{align}
h_2^2((f_{\mathcal{I}}^{\prime}(x_k))^{\dagger})= \min \limits _{\mathcal{I}_{k}} h_2^2((f_{\mathcal{I}_{k}}^{\prime}(x_k))^{\dagger}) \textrm{ and } \sigma^2_{\max}\left((f_{\mathcal{I} }^{\prime}(x_k))^{\dagger}\right)= \max \limits _{\mathcal{I}_{k}}\sigma^2_{\max}\left((f_{\mathcal{I}_{k}}^{\prime}(x_k))^{\dagger}\right), \notag
\end{align}
then the iterations of the RB-CNK method in Algorithm \ref{The RB-CNK method} satisfy
\begin{align}
\left\| {x}_{k+1}- {x}_{\star}\right\|^{2}_2 \leq
\left(1-\beta \lvert\mathcal{I}_{k}\rvert \delta_k\frac{1}{1+\eta^2}h_2^2(f^{\prime} \left(x_{k}\right))\right)
\left\|x_{k}-{x}_{\star}\right\|^{2}_2.\label{th:RB-CNK_0}
 \end{align}
\end{theorem}

\begin{proof}
From Lemma \ref{sec:theory-mutiple-lemma1}, the definition of $\beta$, and Algorithm \ref{The RB-CNK method}, we get
\begin{align}
\left\|x_{k+1}-x_{\star}\right\|^{2}_2
&\leq
\left\|x_{k}-x_{\star}\right\|^{2}_2-  \left(h_2^2((f_{\mathcal{I}_{k}}^{\prime}(x_k))^{\dagger})-2\eta\sigma^2_{\max}\left((f_{\mathcal{I}_{k}}^{\prime}(x_k))^{\dagger}\right)\right)\left\|  f_{\mathcal{I}_{k}}(x_k)\right\|^{2}_2  \notag
\\
&\leq
\left\|x_{k}-x_{\star}\right\|^{2}_2- \beta\left\|  f_{\mathcal{I}_{k}}(x_k)\right\|^{2}_2 \notag
\\
&=
\left\|x_{k}-x_{\star}\right\|^{2}_2-\beta\sum\limits_{j\in\mathcal{I}_{k} } \left\lvert  f_{j}(x_k)\right\rvert^{2} \notag
\\
&\leq
\left\|x_{k}-x_{\star}\right\|^{2}_2-\beta \lvert \mathcal{I}_k \rvert  \delta_k\|f(x_k)\|_2^2, \notag
\end{align}
which together with Lemma \ref{lemma3} yields
\begin{align}
\left\|x_{k+1}-x_{\star}\right\|^{2}_2
&\leq
\left\|x_{k}-x_{\star}\right\|^{2}_2-\beta \lvert \mathcal{I}_k \rvert  \delta_k\frac{1}{1+\eta^2}\left\|f^{\prime} \left(x_{k}\right)\left(x_{k}-x_{\star}\right)\right\|^2_2 \notag
\\
&\leq
\left(1-\beta \lvert \mathcal{I}_k \rvert  \delta_k\frac{1}{1+\eta^2}h_2^2(f^{\prime} \left(x_{k}\right))\right)\left\|x_{k}-x_{\star}\right\|^2_2. \notag
\end{align}
So, the desired result (\ref{th:RB-CNK_0}) is obtained.
\end{proof}

\section{Experimental results}
\label{sec:experiments}
In this section, we mainly compare our new methods, i.e., the DR-CNK, RD-CNK, DB-CNK and RB-CNK methods, with the existing methods for solving
Brown almost linear function and generalized linear model (GLM) in terms of the iteration numbers (denoted as ``IT'') and computing time in seconds (denoted as ``CPU''). The IT and CPU here are respectively the average of the IT and CPU over 10 runs of the algorithm and all experiments terminate once $ \left\|f(x_{k})\right\|_2^2<10^{-6} $ or the number of iterations exceeds 200000.

\subsection{Brown almost linear function}
\label{sec:Brown almost linear function}
The function \cite{more1981testing, wang2022nonlinear} is expressed as follows
$$
\begin{array}{lr}
f_k(x)=x^{(k)}+\sum_{i=1}^n x^{(i)}-(n+1), & 1 \leq k<n;\\
f_k(x)=\left(\prod_{i=1}^n x^{(i)}\right)-1, & k=n.
\end{array}
$$
Here we only compare our methods with the NRK method because the authors in \cite{wang2022nonlinear} have compared the NRK method with other methods in detail and concluded that the NRK method performed better in most cases. All experiments are start at $x_{0}=0.5*ones(n,1)$.

We list the iteration numbers and computing time for the DR-CNK, RD-CNK, DB-CNK, RB-CNK and NRK methods
in Tables \ref{tab_Brown_IT} and \ref{tab_Brown_CPU}. They 
show that our new methods vastly outperform the NRK method. 
In particular, 
the time speedup of the single sample-based methods, i.e., the DR-CNK and RD-CNK methods, against to the NRK method is maintained at about 10 times for most cases, and the time speedup of the multiple samples-based methods, i.e., the DB-CNK and RB-CNK methods, against to the NRK method can even reach 200 times, as can be seen in the last two rows of Table \ref{tab_Brown_CPU}. Overall, the numerical results indicate that our new greedy capped schemes are better than the randomized strategy used in the NRK method.

\begin{table}[tp]
  \centering
   \fontsize{8}{8} \selectfont
    \caption{IT comparison of our methods and the NRK method.}
    \label{tab_Brown_IT}
    \begin{tabular}{cccccc}
    \hline
$m \times n$    &NRK&DR-CNK& RD-CNK&  DB-CNK &  RB-CNK \cr
 \hline
$50\times 50$    &  4780.2  &  755.2   &  755     & 1    &1  \cr
$100\times 100$  &  16218   &  1306    &  1308    & 1    &1  \cr
$150\times 150$  &  33764   &  1902    &  1902    & 1    &1  \cr
$200\times 200$  &  57119   &  2516    &  2506.4  & 1    &1    \cr
$250\times 250$  &  84874   &  3134    &  3128    & 1    &1  \cr
$300\times 300$  &  117100  &  3750.2  &  3750    & 1    &1   \cr
$350\times 350$  &  157190  &  4371.8  &  4371.8  & 1    &1  \cr
$400\times 400$  &  199400  &  4992.4  &  4992.4  & 1    &1   \cr
\hline
\end{tabular}
\end{table}

\begin{table}[tp]
  \centering
   \fontsize{8}{8} \selectfont
    \caption{CPU comparison of our methods and the NRK method.}
    \label{tab_Brown_CPU}
    \begin{tabular}{ccccccc}
    \hline
    $m \times n$    & NRK       &  DR-CNK   & RD-CNK  &  DB-CNK &  RB-CNK \cr
 \hline
$50\times 50$       & 0.2766    & 0.0656    & 0.0578  &  0.0078 &  0.0094 \cr
$100\times 100$     & 0.6078    & 0.1453    & 0.0969  &  0.0172 &  0.0125 \cr
$150\times 150$     & 1.1656    & 0.2141    & 0.1359  &  0.0125 &  0.0281 \cr
$200\times 200$     & 2.0438    & 0.2969    & 0.1953  &  0.0063 &  0.0078 \cr
$250\times 250$     & 3.0078    & 0.3391    & 0.2359  &  0.0078 &  0.0422 \cr
$300\times 300$     & 4.1422    & 0.3984    & 0.2797  &  0.0156 &  0.0156 \cr
$350\times 350$     & 5.7625    & 0.4609    & 0.3359  &  0.0266 &  0.0187 \cr
$400\times 400$     & 7.5219    & 0.4922    & 0.3922  &  0.0500 &  0.0375 \cr
\hline
\end{tabular}
\end{table}

\subsection{GLM}
\label{sec:logistic regression}
The regularized GLM has the form
$$ \min \limits _{w \in \mathbb{R}^d} P(w) \stackrel{\text { def }}{=} \frac{1}{p} \sum_{i=1}^p \phi_i\left(a_i^{\top} w\right)+\frac{\lambda}{2}\|w\|^2,$$ where $\phi_i(t)=\ln \left(1+\mathrm{e}^{-y_i t}\right)$ is the logistic loss, $y_i\in\{-1, 1\}$ is the $i$th target value, $a_i \in \mathbb{R}^d$ is $i$th data sample, and $w \in \mathbb{R}^d$ is the parameter to optimize. By adopting the equivalent transformation discussed in \cite{yuan2022sketched}, we can get the following nonlinear problem:
$$
f(x) \stackrel{\text { def }}{=}\left[\begin{array}{c}
\frac{1}{\lambda p} A \alpha-w \\
\alpha+\Phi(w)
\end{array}\right]=0
,$$
where $f: \mathbb{R}^{p+d} \rightarrow \mathbb{R}^{p+d}$, $x=\left[\begin{array}{c}
 \alpha \\
w
\end{array}\right]\in \mathbb{R}^{p+d} $, $A \stackrel{\text { def }}{=}\left[\begin{array}{lll}a_1, & \cdots, & a_p\end{array}\right] \in \mathbb{R}^{d \times p}$, $\Phi(w) \stackrel{\text { def }}{=}\left[\phi_1^{\prime}\left(a_1^{T} w\right), \cdots, \phi_p^{\prime}\left(a_p^{T} w\right)\right]^{T} \in \mathbb{R}^p$ and $\alpha=-\Phi(w)$. For the problem, we only compare our methods with the sketched Newton-Raphson (SNR) \cite{yuan2022sketched} to further illustrate the advantages of our greedy capped sampling over uniform sampling.

We use datasets in the experiments for GLM taken from \cite{chang2011libsvm} on \url{https://www.csie.ntu.edu.tw/~cjlin/libsvmtools/datasets/} and the scaled versions are applied if provided. They have disparate properties, either ill or well conditioned, dense or sparse; see details in Table \ref{tab_properties of dataset}. Note that in Table \ref{tab_properties of dataset}, C.N is the condition number of the data matrix $A$, the smoothness constant $L\stackrel{\text { def }}{=}\frac{ \lambda_{\max}(AA^T)}{4p}+\lambda$ and $$\text{density}\stackrel{\text { def }}{=}\frac{\text{number of nonzero of a $d\times p$ data matrix $A$}}{dp}.$$ For all methods, we use $\lambda=\frac{1}{p}$ as the regularization parameter and let the initial value $x_0$ be zero, i.e., $w_0 =0\in \mathbb{R}^d$ and $\alpha_0 =0\in \mathbb{R}^p$.

\begin{table}[]
\centering
   \fontsize{8}{8}\selectfont
       \caption{Details of the data sets for GLM.}
    \label{tab_properties of dataset}
    \begin{tabular}{ c| c c c c  c }
 \hline
   dataset       &dimension $(d)$   &  samples $(p)$    & $L$           & C.N               &density\cr \hline
   fourclass     &     2            &  862              & 0.0824        &  1.0757           &0.9959\cr
   german.numer  &    24            & 1000              &  2.1113       & 15.4082           &0.9584\cr
   heart         &    13            &  270              & 0.6973        &  7.0996           &0.9624\cr
   ionosphere    &    34            &  351              &1.5290         & 2.4485e+17        &0.8841\cr
   diabetes      &    8             & 768               & 0.5740        &  8.2105           &0.9985\cr
   sonar         &    60            & 208               & 3.2282        &  89.9388          &0.9999\cr
   w1a           &    300           & 2477              & 0.6224        &  3.4303e+34       &0.0382\cr
   w2a           &    300           & 3470              & 0.6347        &  1.0416e+34       &0.0388\cr
   w3a           &    300           & 4912              & 0.6436        &3.9777e+33         &0.0388\cr
   \hline
\end{tabular}
\end{table}

For the SNR method, different parameter settings will lead to different methods, in which the Kaczmarz-TCS and Block TCS methods are discussed in detail as the main representatives \cite{yuan2020sketched,yuan2022sketched}. They 
are single-sample and multi-sample methods, respectively. So, we will compare the single-sample methods, i.e., the DR-CNK, RD-CNK and Kaczmarz-TCS methods, and the multi-sample methods, i.e., the DB-CNK, RB-CNK and Block TCS methods,  separately. 

The parameters used in the example are
$\gamma=1$, $\tau_d=d$, and $\tau_p=150$, and the Bernoulli parameter $b=\frac{p}{(p+\tau_p)}-0.11$. In the DB-CNK and RB-CNK methods, we adopt the same iterative framework as the Block TCS method. Specifically, the least norm solution is computed directly for the first $d$ rows of $f$, while the greedy capped strategies in Algorithms \ref{The DB-CNK method} and \ref{The RB-CNK method} are used for the last $p$ rows of $f$. In doing so, we can not only use the structure of the nonlinear function $f$, but also directly compare the relationship between greedy capped sampling and uniform sampling.

We show the results of our methods compared to the SNR method, i.e., the Kaczmarz-TCS and Block TCS methods, on GLM in Figures \ref{fig_single_sample_IT} to \ref{fig_multi_samples}. In  the first two figures, 
we test six datasets for single-sample methods and find that the performance of our DR-CNK and RD-CNK methods is not much different, but in most cases, the former is slightly better than the latter in terms of iteration numbers and computing time, and both methods are more efficient than the Kaczmarz-TCS method. 
In Figure \ref{fig_multi_samples}, we test three datasets for multi-sample methods and get that our RB-CNK and DB-CNK methods perform about the same, and both methods outperform the Block TCS method. All these figures imply that the new greedy capped strategies are feasible and outperform the uniform sampling.

\begin{figure}[ht]
 \begin{center}
\includegraphics [height=3.5cm,width=3.8cm ]{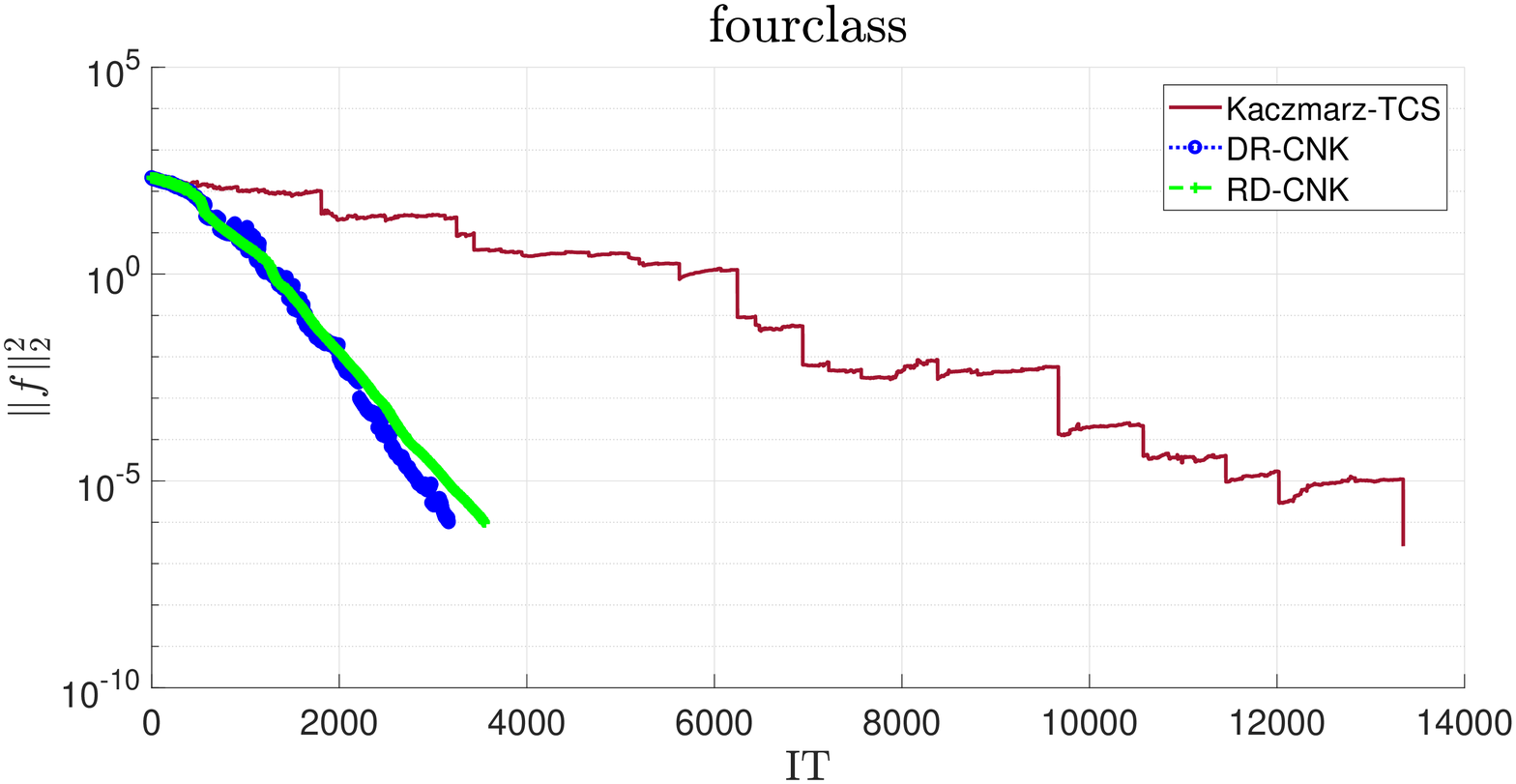}
\includegraphics [height=3.5cm,width=3.8cm ]{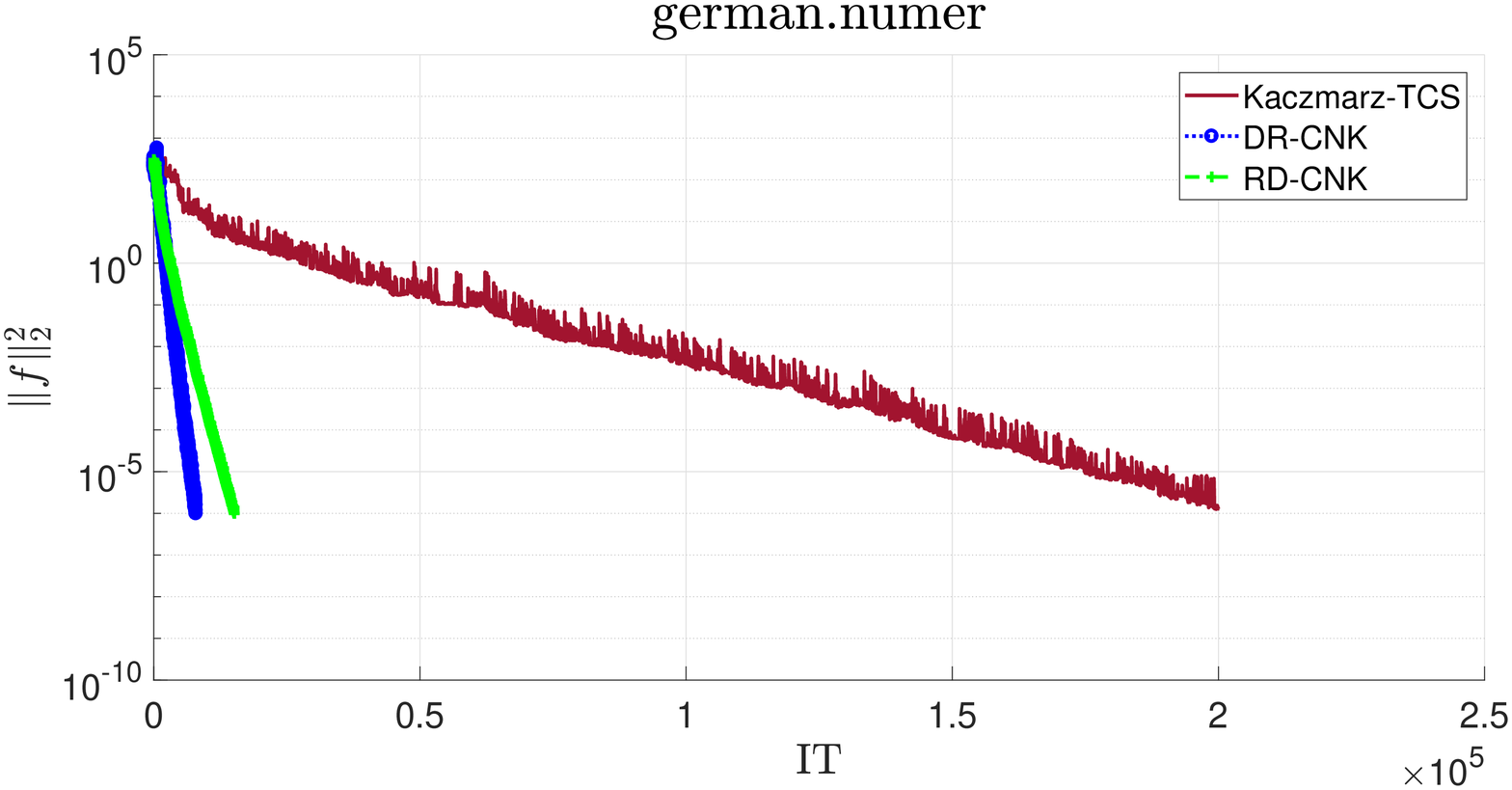}
\includegraphics [height=3.5cm,width=3.8cm ]{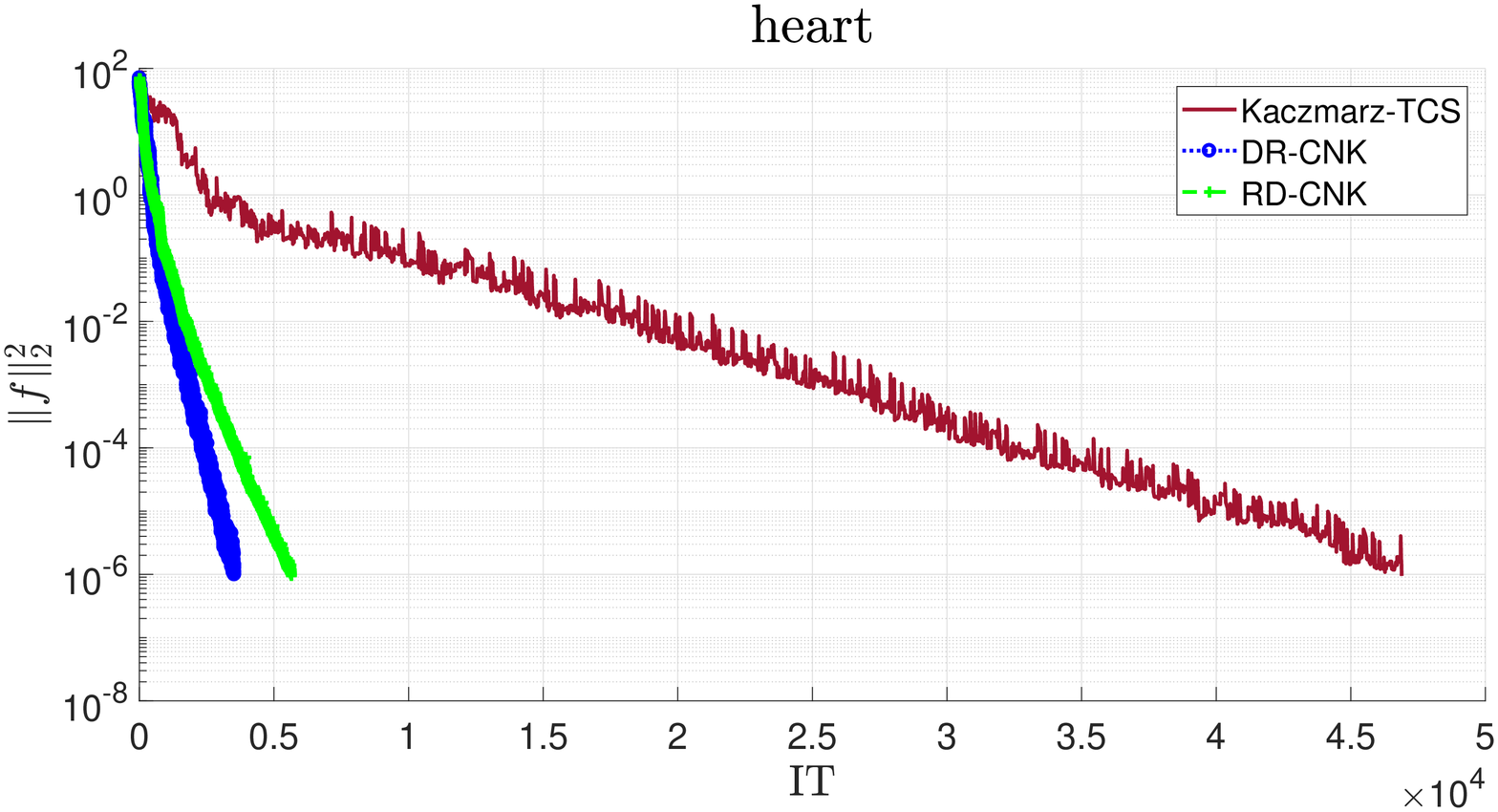}
\includegraphics [height=3.5cm,width=3.8cm ]{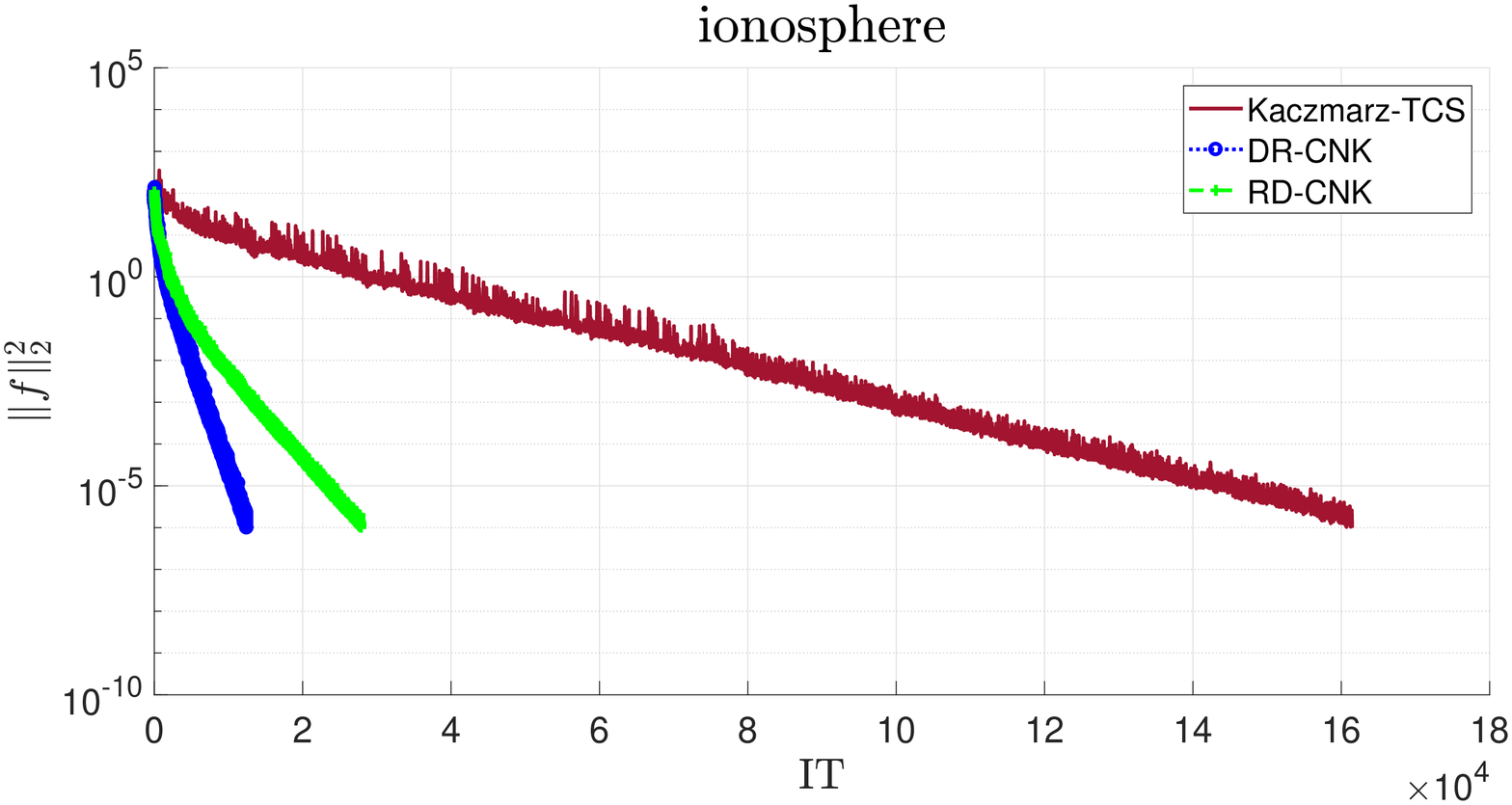}
\includegraphics [height=3.5cm,width=3.8cm ]{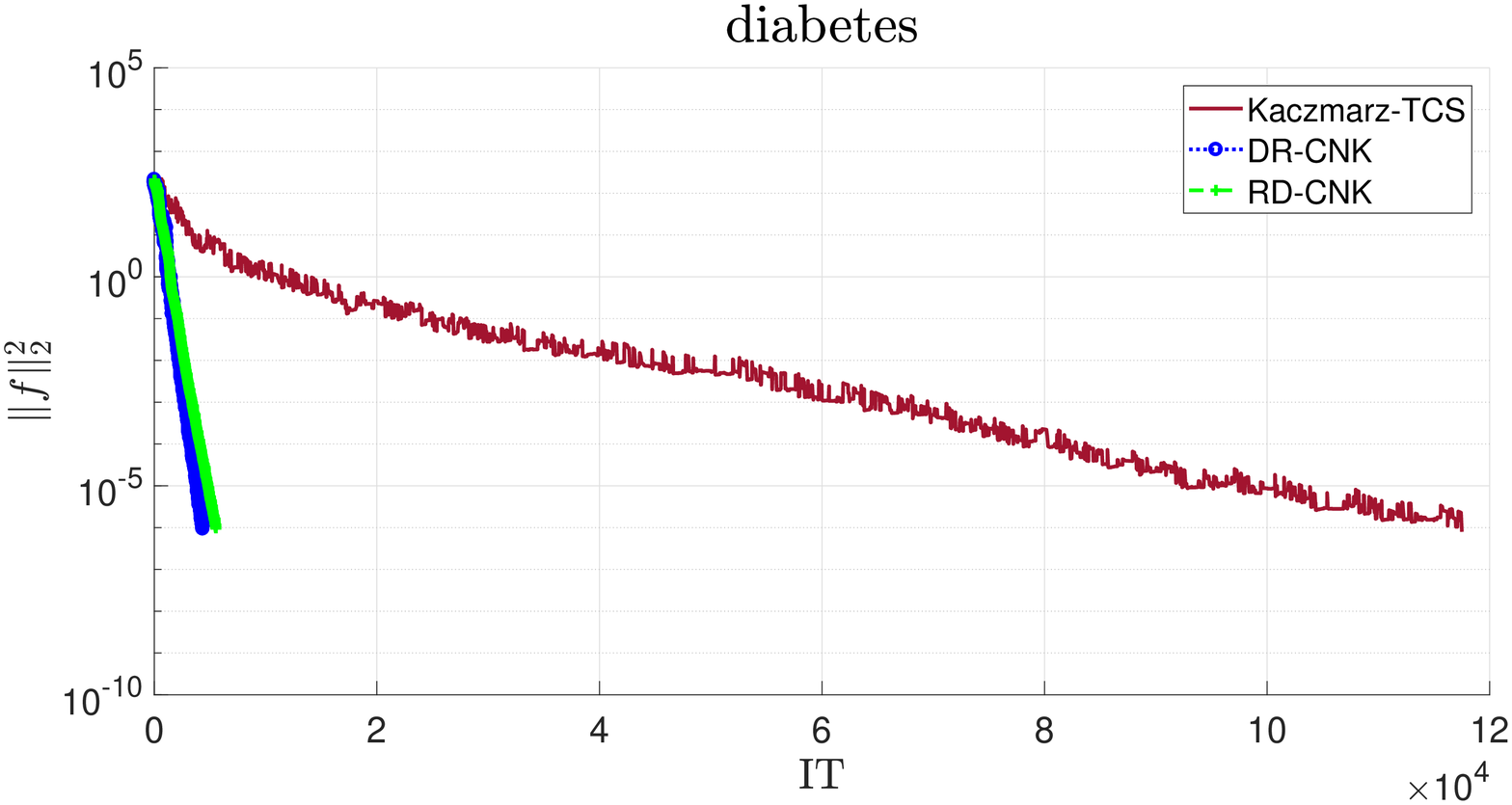}
\includegraphics [height=3.5cm,width=3.8cm ]{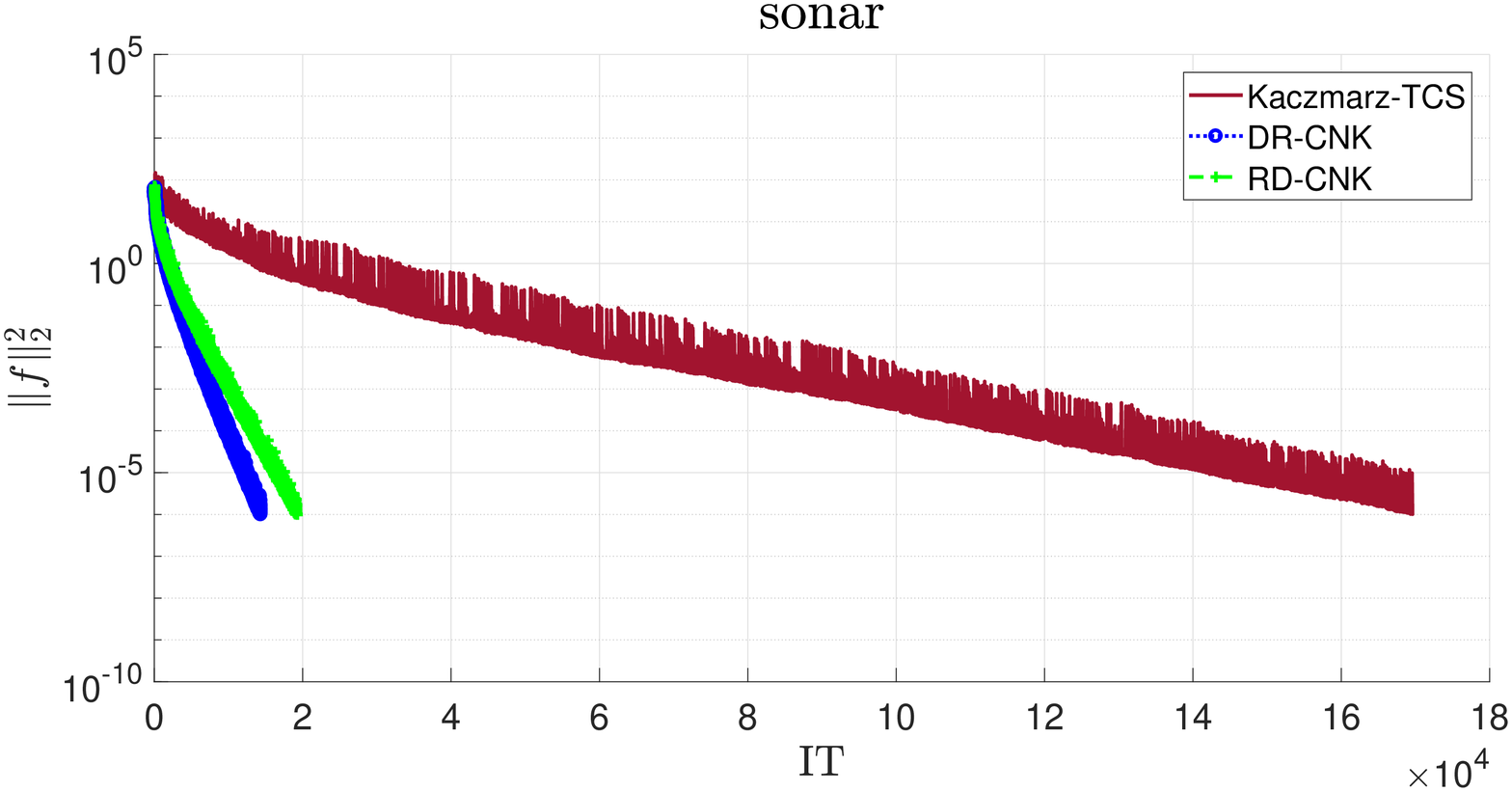}
 \end{center}
\caption{$\|f\|_2^2$ versus IT for the Kaczmarz-TCS, DR-CNK and RD-CNK methods.}\label{fig_single_sample_IT}
\end{figure}
\begin{figure}[ht]
 \begin{center}
\includegraphics [height=3.5cm,width=3.8cm ]{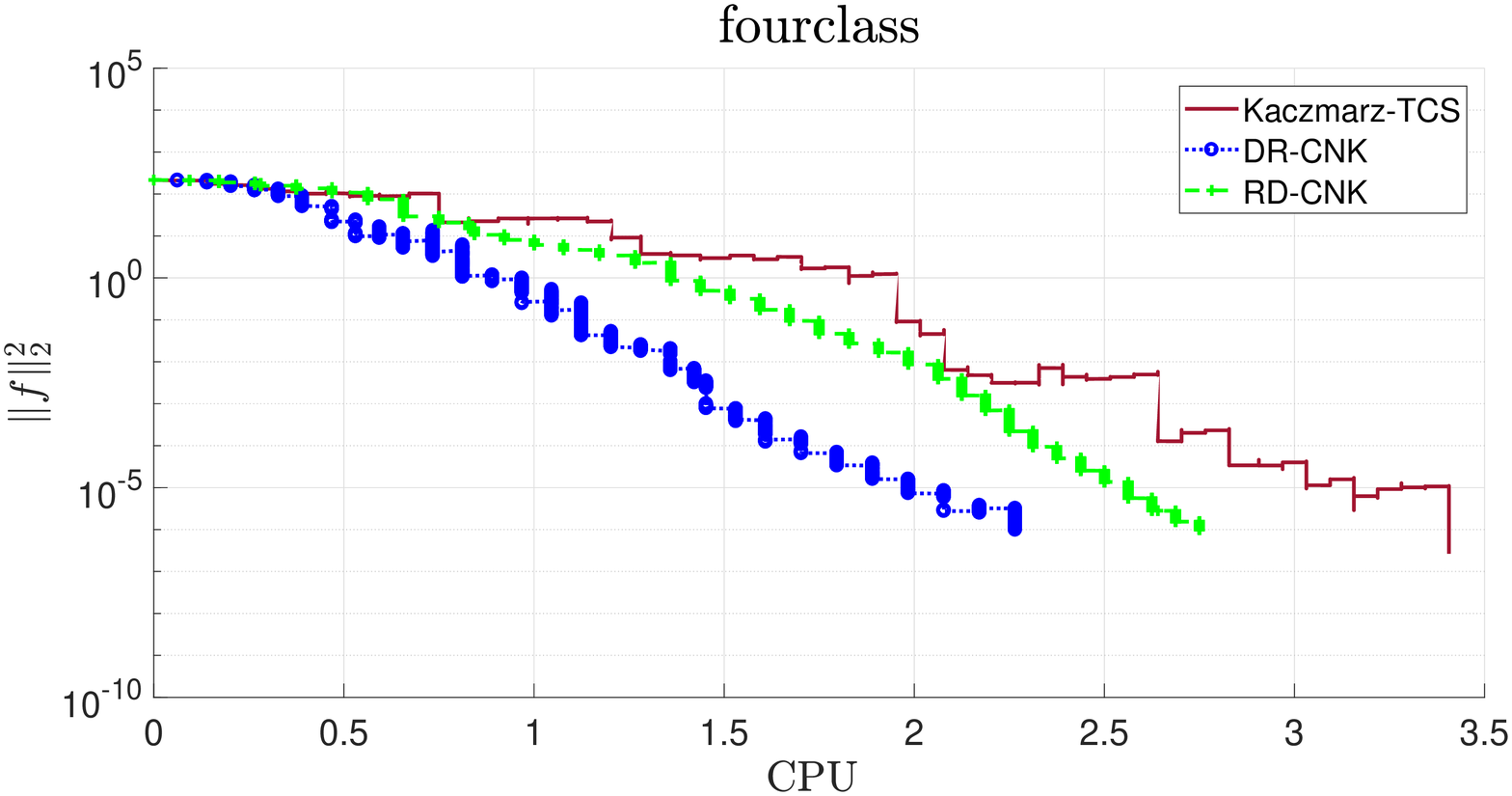}
\includegraphics [height=3.5cm,width=3.8cm ]{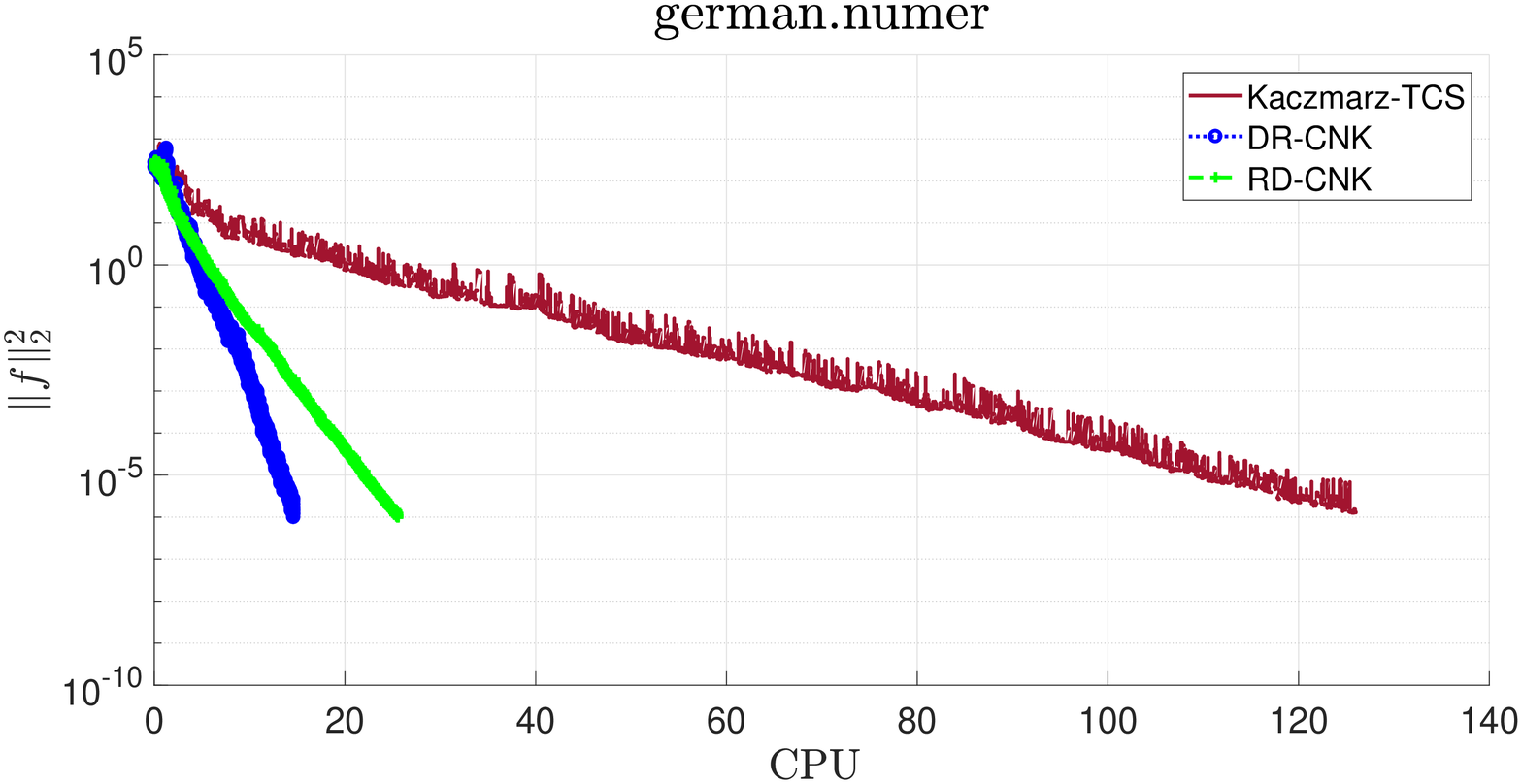}
\includegraphics [height=3.5cm,width=3.8cm ]{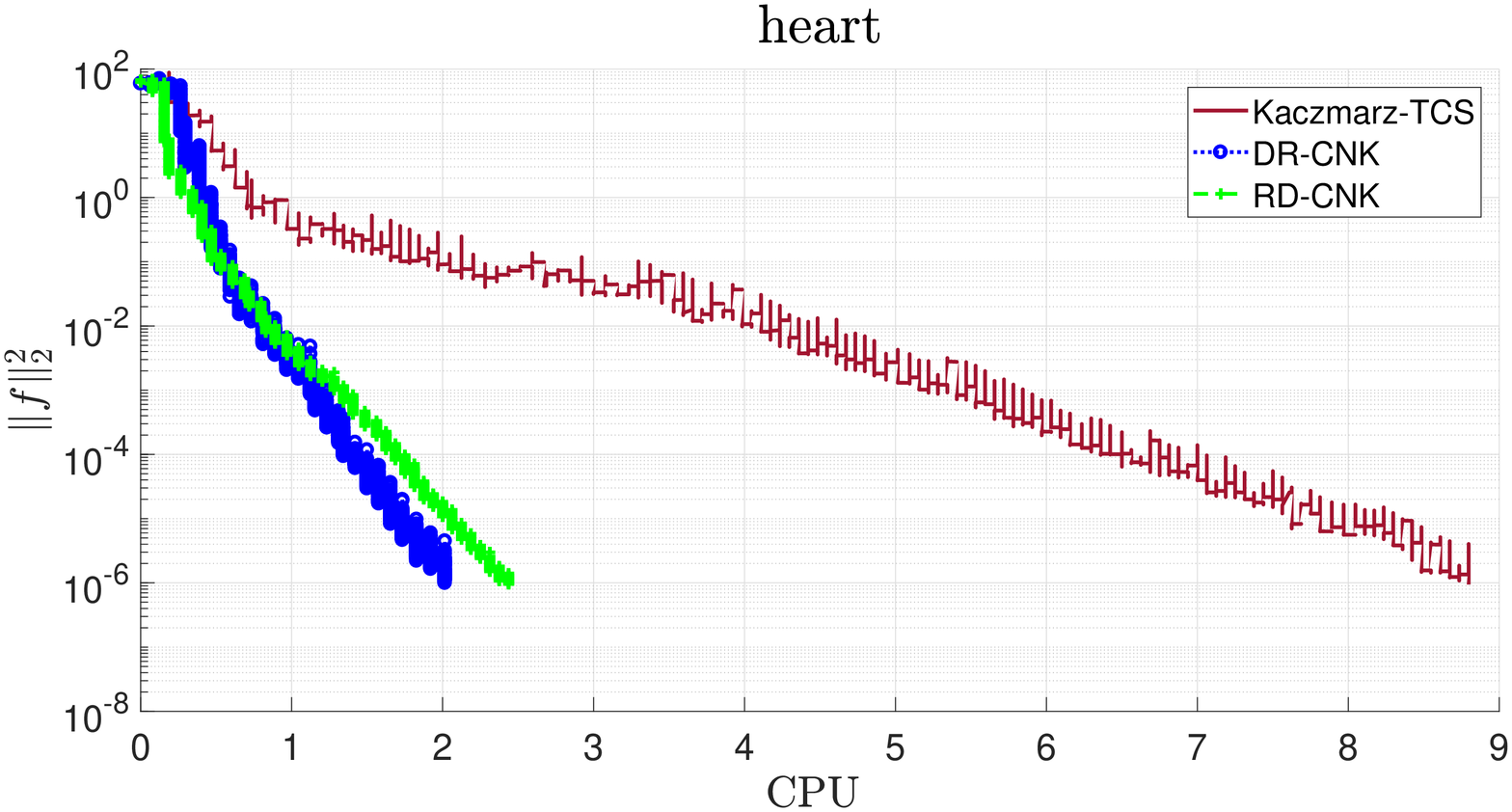}
\includegraphics [height=3.5cm,width=3.8cm ]{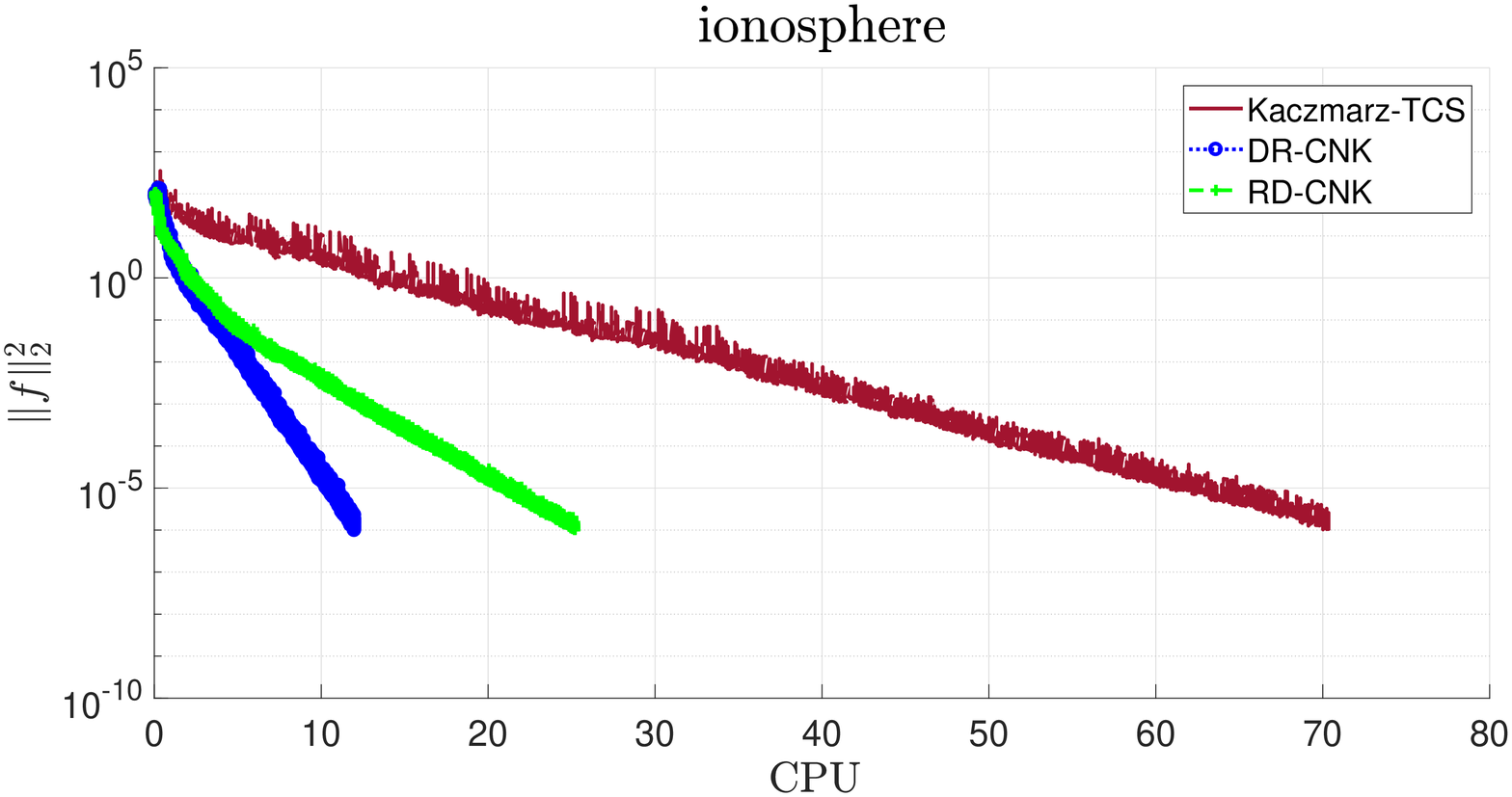}
\includegraphics [height=3.5cm,width=3.8cm ]{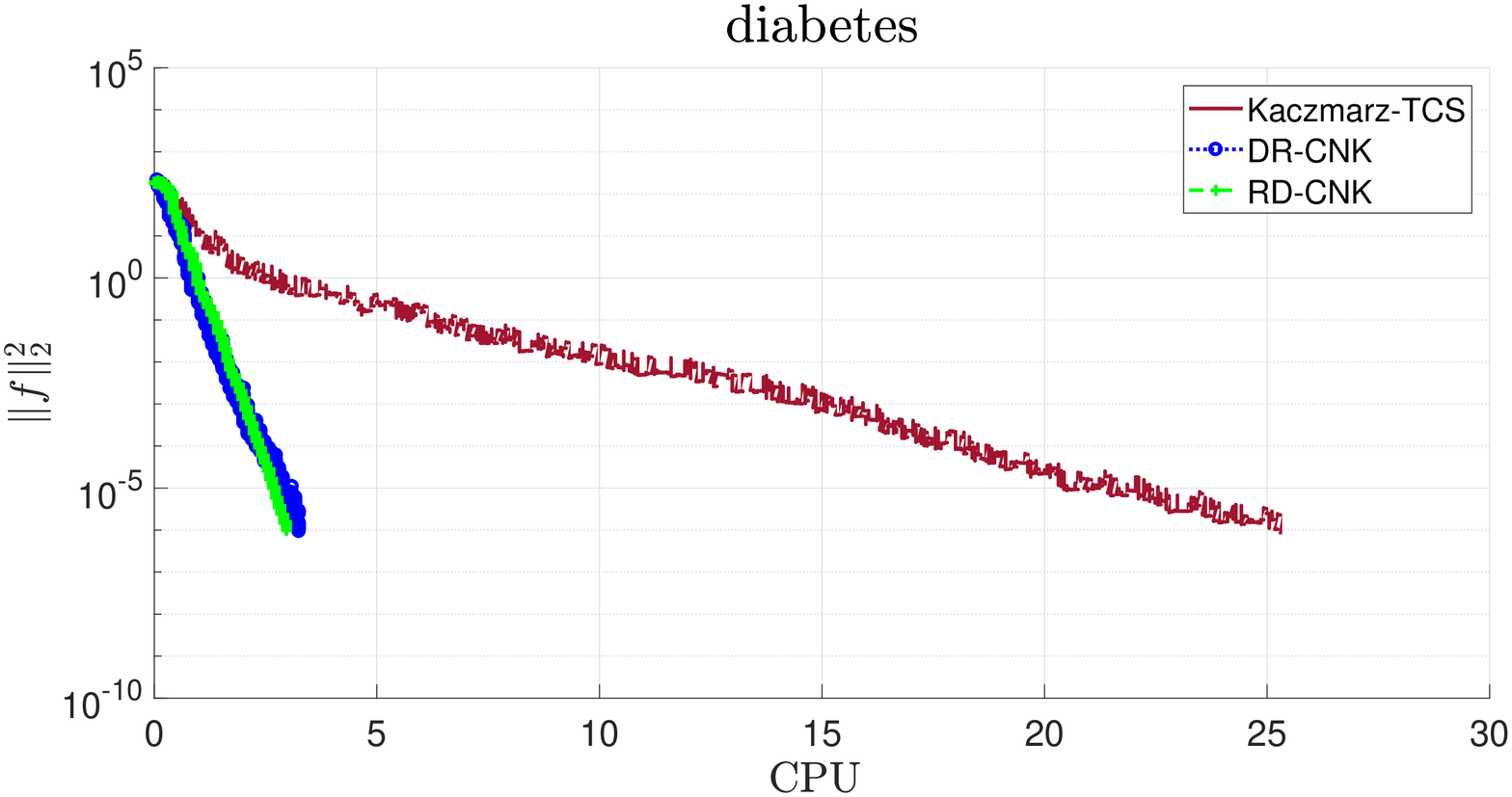}
\includegraphics [height=3.5cm,width=3.8cm ]{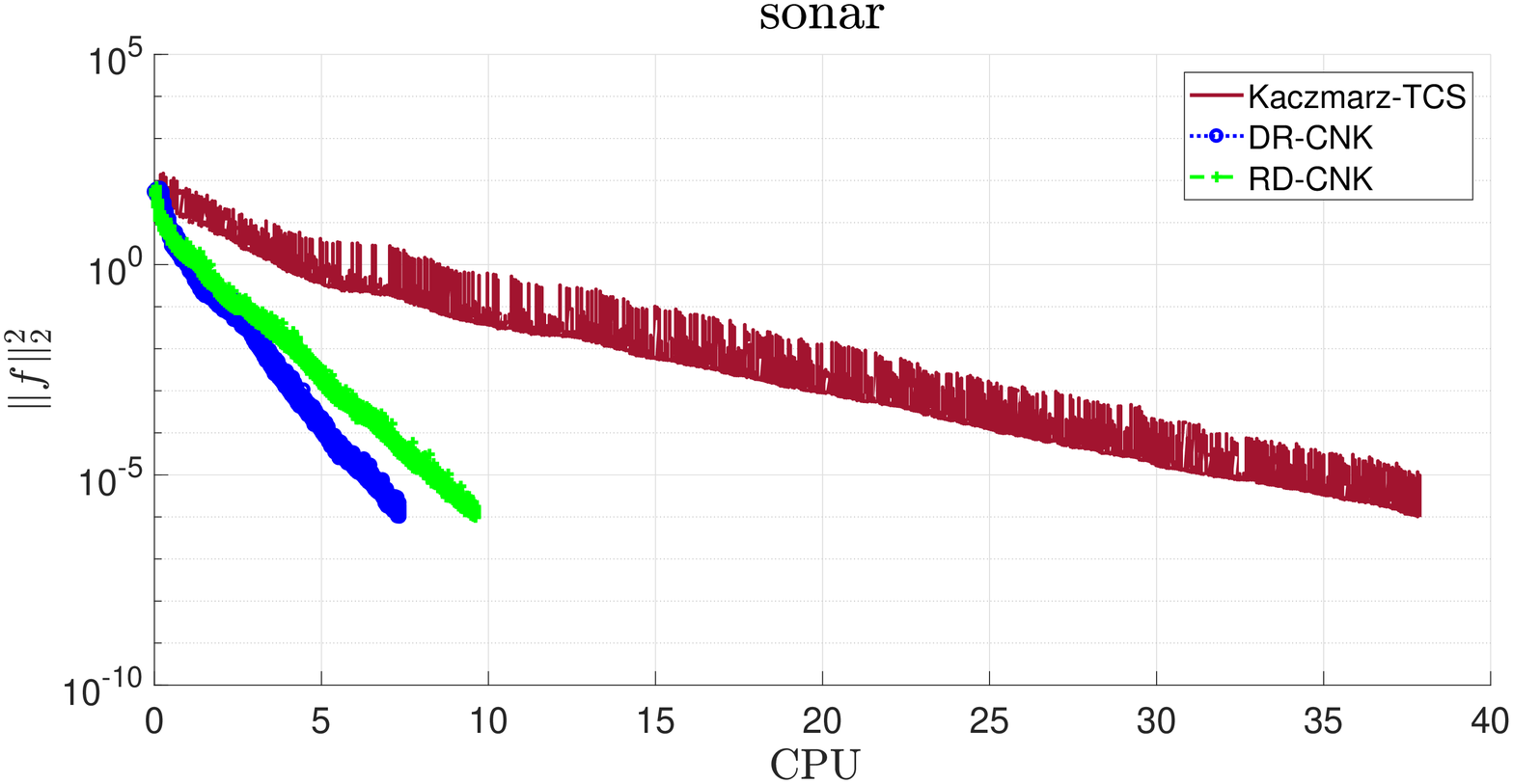}
 \end{center}
\caption{$\|f\|_2^2$ versus CPU for the Kaczmarz-TCS, DR-CNK and RD-CNK methods.}\label{fig_single_sample_CPU}
\end{figure}
\begin{figure}[ht]
 \begin{center}
\includegraphics [height=3.5cm,width=3.8cm ]{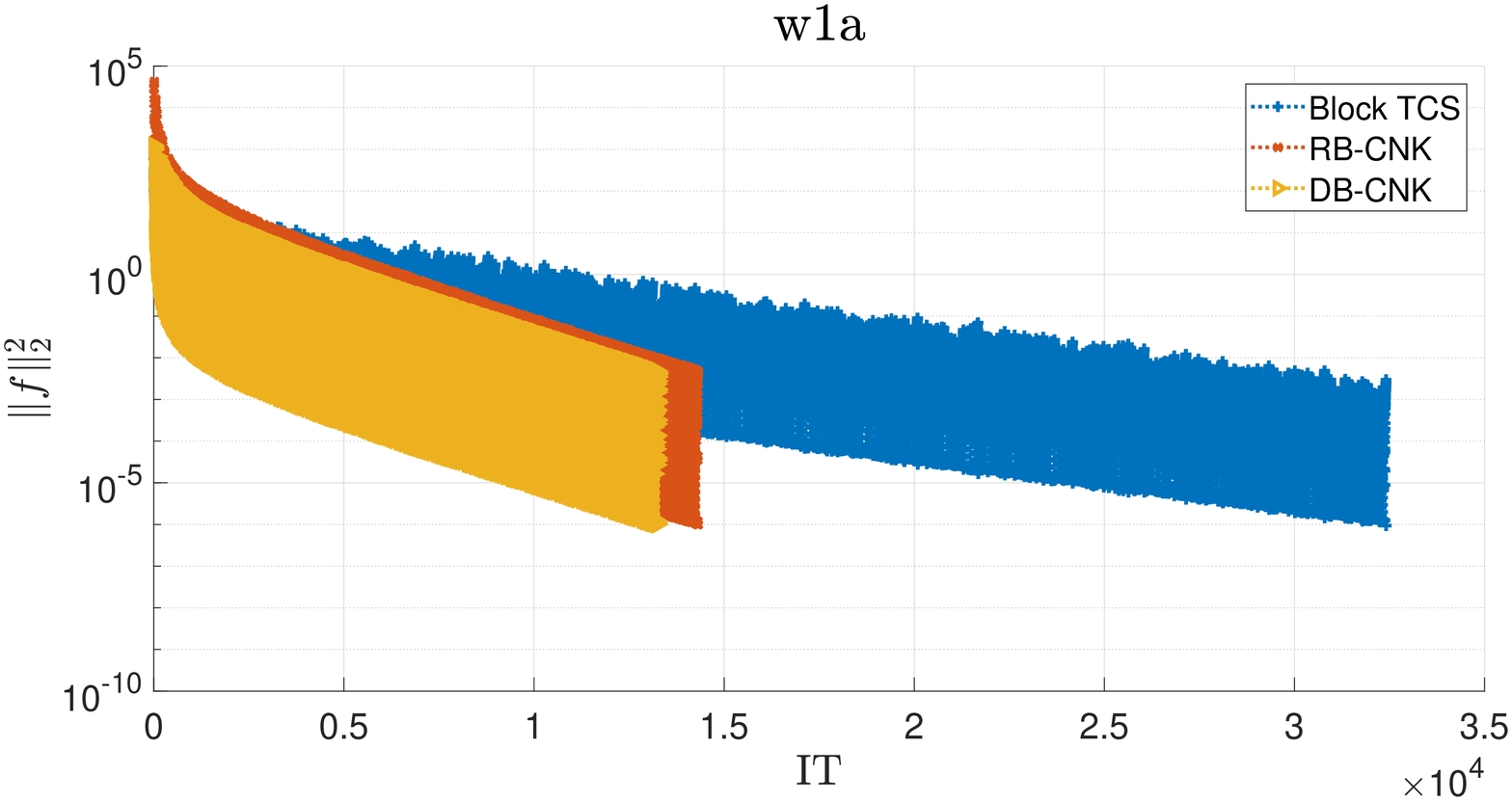}
\includegraphics [height=3.5cm,width=3.8cm ]{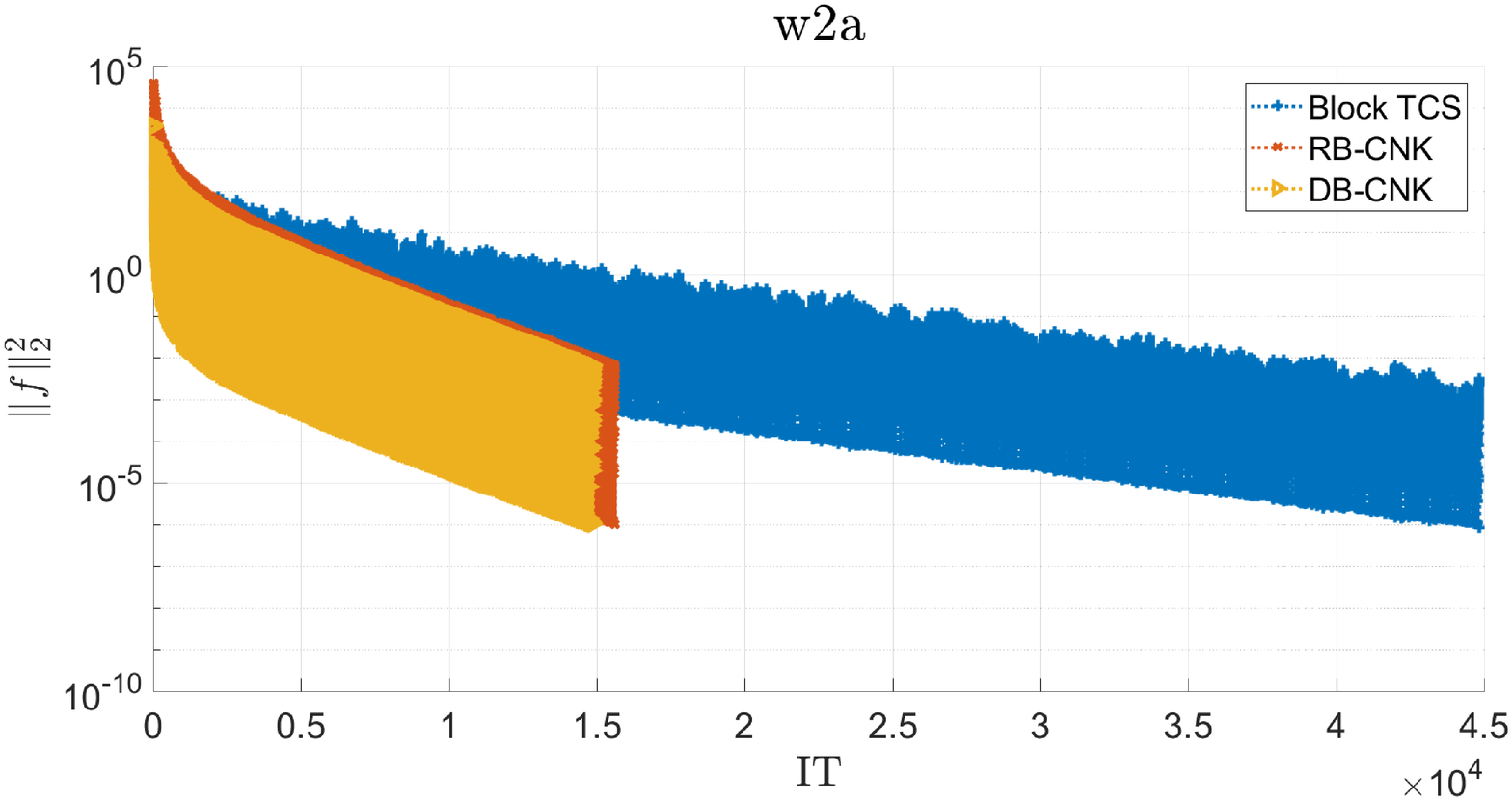}
\includegraphics [height=3.5cm,width=3.8cm ]{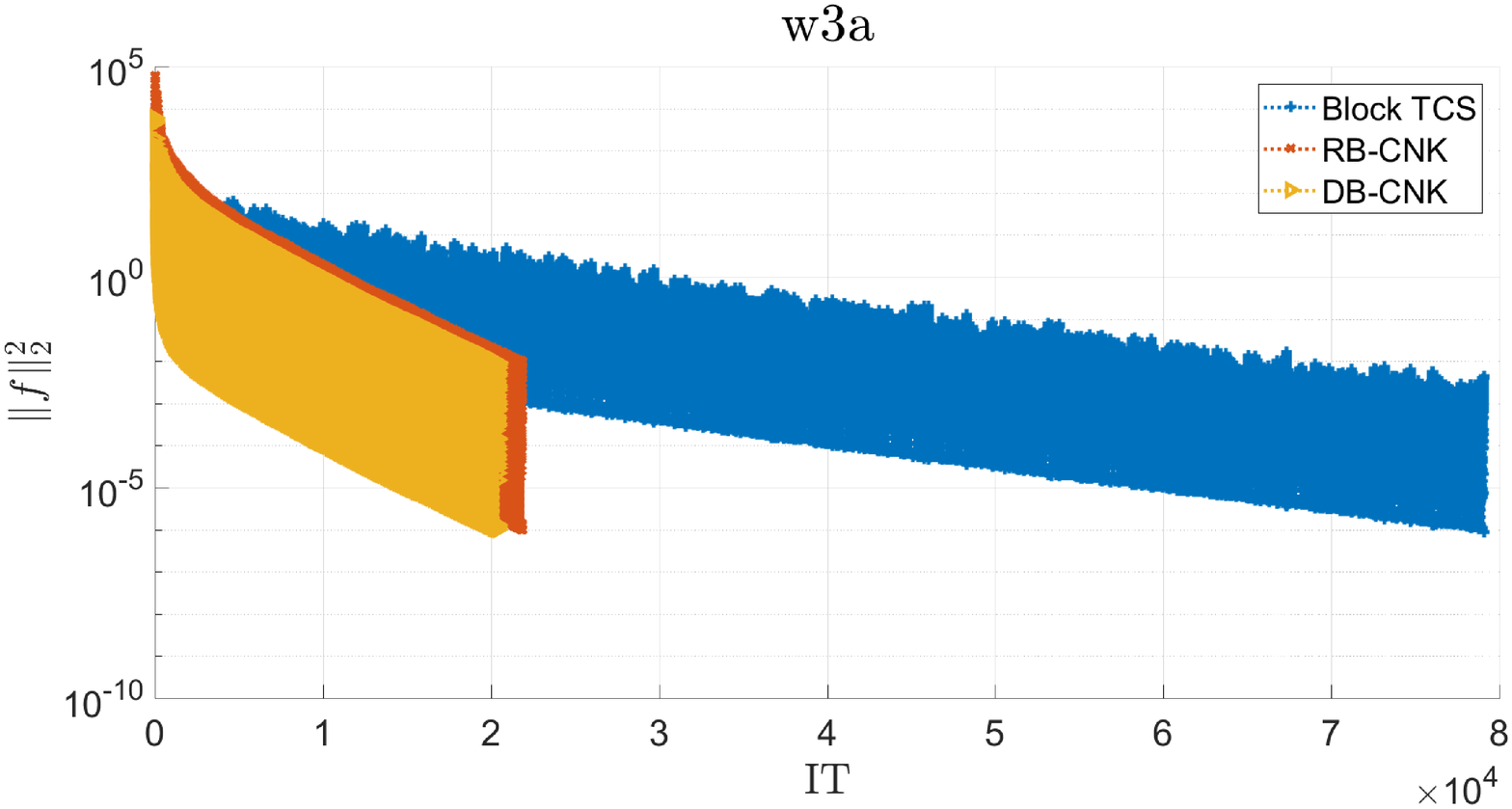}
\includegraphics [height=3.5cm,width=3.8cm ]{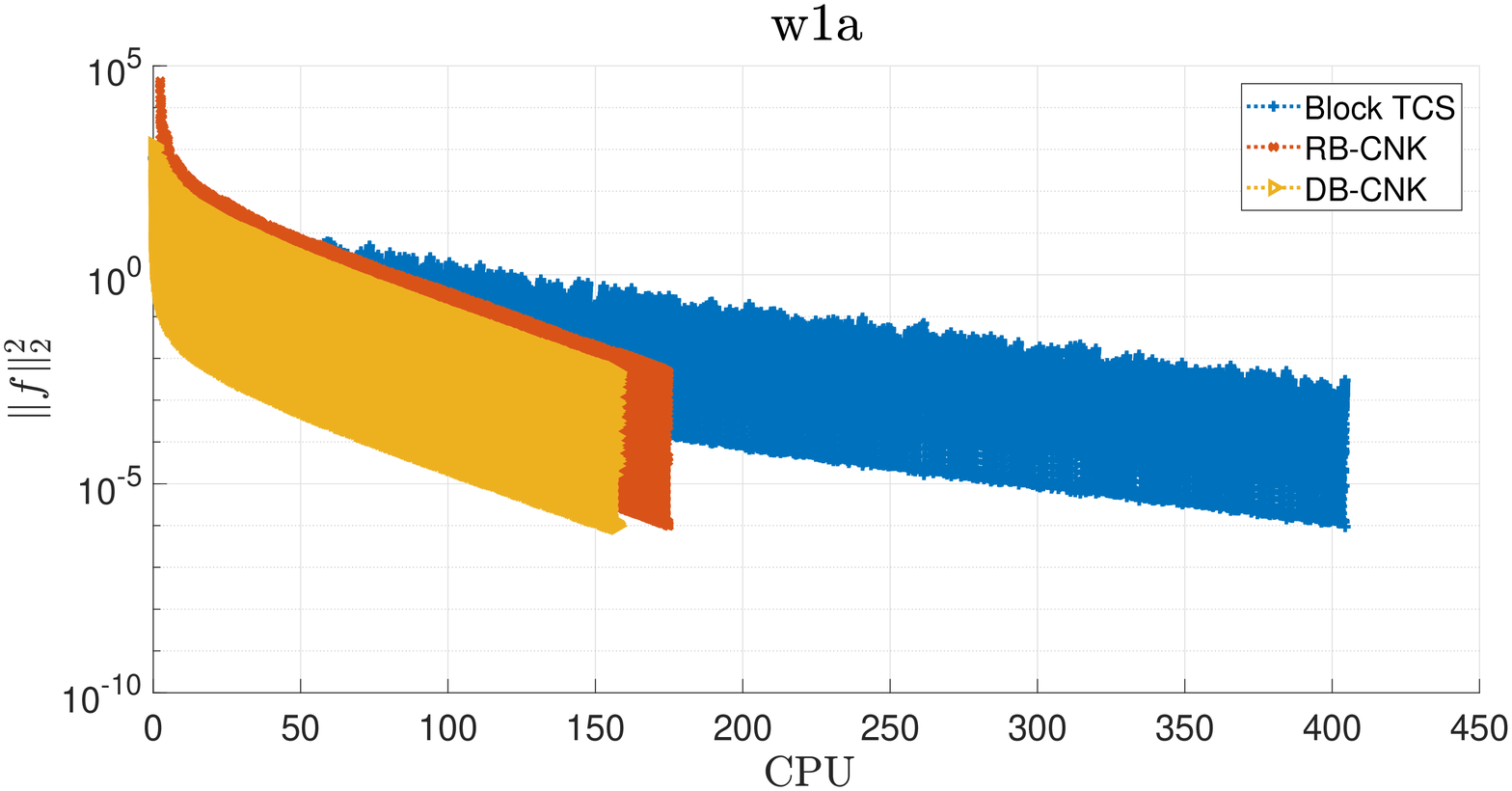}
\includegraphics [height=3.5cm,width=3.8cm ]{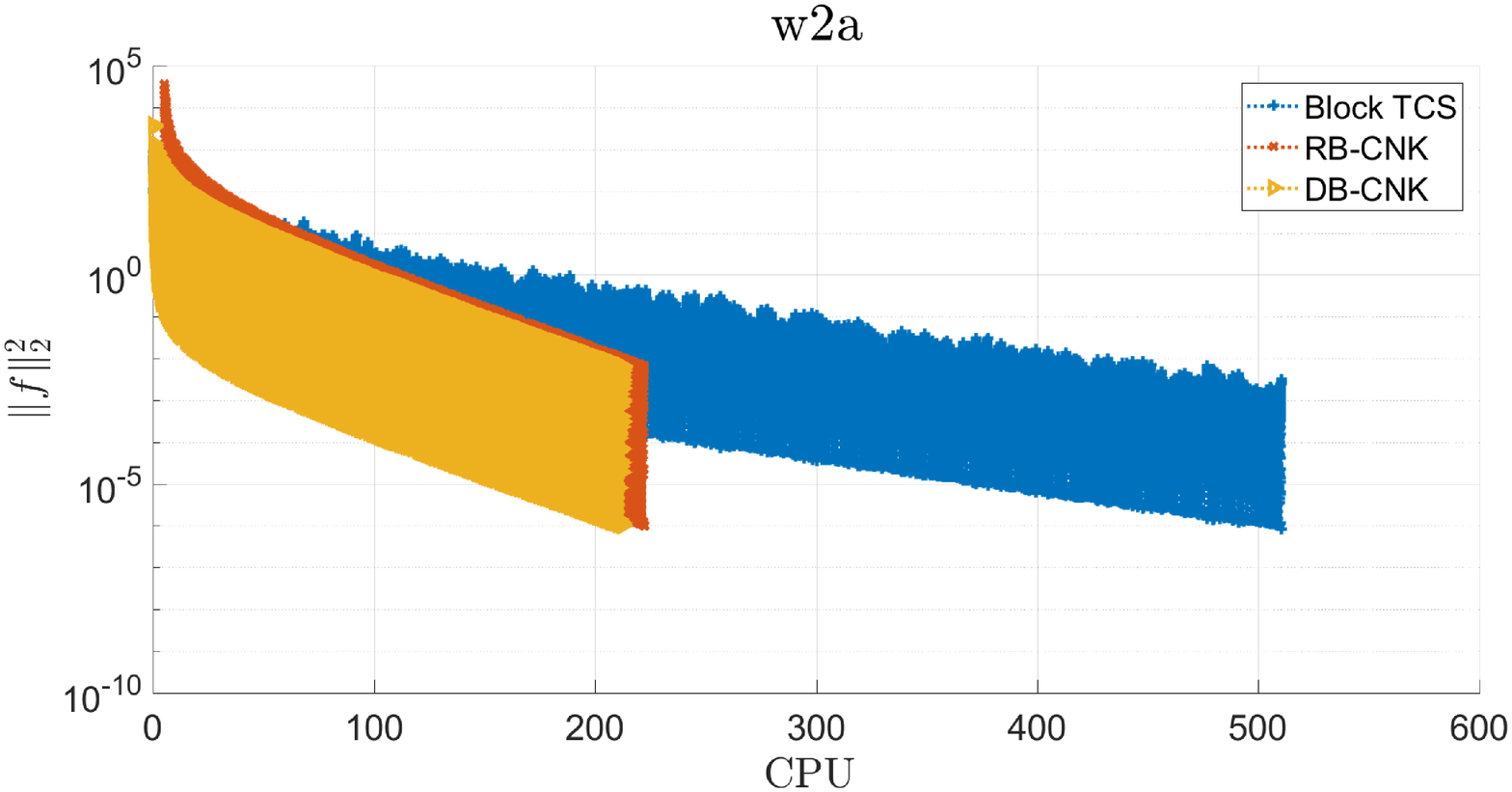}
\includegraphics [height=3.5cm,width=3.8cm ]{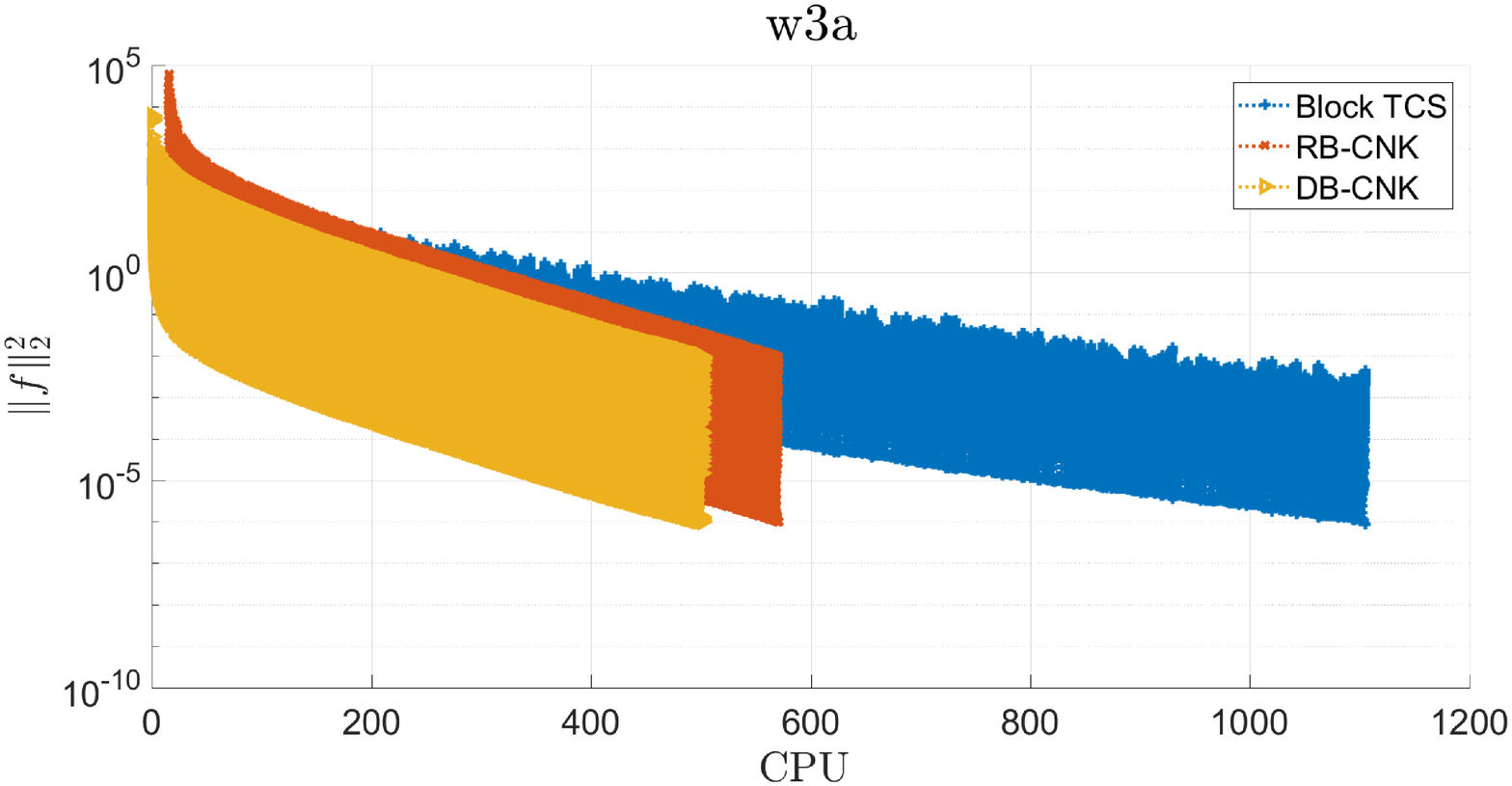}
 \end{center}
\caption{$\|f\|_2^2$ versus IT and CPU for the Block TCS, RB-CNK and DB-CNK methods.}\label{fig_multi_samples}
\end{figure}

\section{Concluding remarks}
\label{sec:conclusions}

This paper proposes two greedy capped nonlinear Kaczmarz methods, i.e., the DR-CNK and RD-CNK methods, which make up for the deficiency of the NRK method, that is, the new methods will definitely not extract the index corresponding to the small component. The theoretical analysis shows that the convergence factors of our two methods are strictly smaller than that of the NRK method, which is also verified by a large number of numerical examples. Further, we present their block versions for acceleration.

Considering the efficiency of the capped threshold, other threshold strategies can be further explored. In addition, although the new block methods perform well, we cannot determine the specific size of the index set in each iteration, which may lead to extreme cases. That is, either there is only one index in the set, or all indices, which is inconsistent with the original intention of the block sampling iteration. Therefore, how to design a more reasonable index set is also worth further discussion.
%

\section*{Declarations}

\bmhead{Ethical Approval}
Not Applicable

\bmhead{Availability of supporting data} The data that support the findings of this study are available from the corresponding author upon reasonable request.

\bmhead{Competing Interests}
The authors declare that they have no conflict of interest.

\bmhead{Funding}
This work was supported by the National Natural Science Foundation of China (No. 11671060) and the Natural Science Foundation Project of CQ CSTC (No. cstc2019jcyj-msxmX0267)
\bmhead{Authors' contributions}
All the authors contributed to the study conception and design, and read and
approved the final manuscript.
\bmhead{Acknowledgments}
Not Applicable

\bibliography{sn-bibliography}


\end{document}